\documentclass[11pt]{amsart}
\usepackage{amssymb,graphicx,color,amsmath} 

\usepackage{enumerate}
\usepackage{hyperref}
\hypersetup{
  pdfborder={0 0 0},
  colorlinks   = true, 
  urlcolor     = blue, 
  linkcolor    = red, 
  citecolor   = blue 
}

\newtheorem{thm}{Theorem}[section]
\newtheorem{lemma}[thm]{Lemma}
\newtheorem{prop}[thm]{Proposition}

\theoremstyle{definition}

\theoremstyle{remark}

\newtheorem{rem}[thm]{Remark}

\newcommand{\Z}{\mathbb{Z}}
\newcommand{\N}{\mathbb{N}}
\newcommand{\R}{\mathbb{R}}
\newcommand{\Q}{\mathbb{Q}}

\newcommand{\G}{\mathcal{G}}

\newcommand{\f}{\varphi}

\newcommand{\St}{{\rm St}}
\newcommand{\alt}{{\rm Alt}}

\newcommand{\be}{\bar e}

\renewcommand{\d}{{\rm d}}

\DeclareMathOperator{\Fix}{Fix}

\newcounter{desccount}
\newcommand{\descitem}[1]{%
  \bf \item[#1] \refstepcounter{desccount}\label{#1}
}
\newcommand{\descref}[1]{\hyperref[#1]{#1}}

\begin{document}

\title{New examples of groups acting on real trees}

\address{Mathematical Sciences,
University of Southampton, Highfield, Southampton, SO17 1BJ, United Kingdom.}

\author{Ashot Minasyan}

\email{aminasyan@gmail.com}

\subjclass[2010]{20E08, 20F65.}
\keywords{Groups acting on trees, fixed point properties, group splittings.}

\begin{abstract}
We construct the first example of a finitely generated group which has Serre's property (FA) (i.e., whenever it acts on a simplicial tree it fixes a vertex),
but admits a fixed point-free action on an $\R$-tree with finite arc stabilizers.
We also give a short and elementary construction of finitely generated groups that have property (FA) but do  not have (F$\R$).
\end{abstract}
\maketitle

\section{Introduction}
In the 1970's  Bass and Serre developed the theory of groups acting on simplicial trees (see \cite{Serre}). In particular, they proved that if a finitely generated group $G$ acts on a simplicial tree non-trivially (i.e., without a global fixed point)
and without edge inversions, then $G$ splits as the fundamental group of a finite graph of groups, where vertex groups are proper subgroups of $G$. Since then there has been a lot of interest in establishing similar
results for actions on more general (non-simplicial) trees. For example, Gillet and Shalen \cite{Gillet-Shalen} 
proved that if $\Lambda$ is a subgroup of $\R$ of $\mathbb{Q}$-rank $1$, then any finitely presented
group admitting a non-trivial action without inversions on a $\Lambda$-tree splits (as an amalgamated product or an HNN-extension) over a proper subgroup.

In the case when $\Lambda=\R$, the first breakthrough was due to Rips, who laid the foundation for the theory of groups acting on $\R$-trees. In particular, he proved that if a finitely presented group $G$ admits a free isometric action
on an $\R$-tree then $G$ is isomorphic to the free product of free abelian and surface groups. Even though Rips never published his work on this topic,
two different proofs of Rips's theorem for finitely generated groups appear in the paper of Bestvina and Feighn \cite{BF} and in the paper \cite{GLP} of
Gaboriau, Levitt and Paulin.

In \cite{BF} Bestvina and Feighn generalized Rips's theory to cover non-free actions. More precisely, they proved that if a finitely presented group $G$ acts non-trivially and stably on
some $\R$-tree $T$, then $G$ splits over an extension $E$-by-finitely generated abelian group, where $E$ fixes an arc of $T$. Here the action is called \emph{stable} if every non-degenerate subtree $S$ of
$T$ contains a non-degenerate subtree
$S' \subseteq S$ such that the pointwise stabilizer $\St_G(S'')$, of any non-degenerate subtree $S'' \subseteq S'$, coincides with the pointwise stabilizer $\St_G(S')$, of $S'$ in $G$
(e.g., this happens if for any descending chain of arcs in $A_1 \supseteq A_2 \supseteq \dots$ in $T$, there is $N \in \N$ such that $\St_G(A_i)=\St_G(A_j)$ for all $i,j \ge N$).

The next important contribution to this theory was made by Sela \cite{Sela}. He showed that if
a freely indecomposable finitely generated group $G$ acts non-trivially and super-stably on an $\R$-tree $T$ with trivial tripod stabilizers
then $T$ has a particular structure and $G$ has an associated decomposition as a fundamental group of a graph of groups
(for the definition of super-stability see \cite[p. 160]{Guirardel}).
In a more recent work \cite{Guirardel}, Guirardel gave an example showing that super-stability is a necessary assumption in Sela's theorem;
he also generalized this result by substituting some of its assumptions with weaker ones.

As the above results show, in many cases the existence of a non-trivial action of a group $G$ on an $\R$-tree (or a $\Lambda$-tree) $T$ implies that $G$ has a non-trivial splitting,
and thus it acts non-trivially on the simplicial Bass-Serre tree associated to this splitting. In fact, Shalen \cite{S} asked whether this is true in general, i.e., if every finitely generated group admitting a non-trivial action on
an $\R$-tree also admits a non-trivial action (without edge inversions) on some simplicial tree. A clear indication that the answer to this question should be negative was
given by Dunwoody in \cite{Dun-small_unstable}, who constructed an example of a finitely generated group which has a non-trivial unstable action on some $\R$-tree with finite cyclic
arc stabilizers, but cannot act non-trivially on a simplicial tree with small edge stabilizers (although, as observed in \cite{Dun-small_unstable},
this group does possess a non-trivial action on some simplicial tree).

Recall that a group $G$ is said to have Serre's property (FA) if any simplicial action of $G$ on a simplicial tree (by isometries and without edge inversions) fixes a vertex;
similarly, $G$ has property (F$\R$) if it cannot act non-trivially on any $\R$-tree.
Clearly (F$\R$) implies (FA), and Shalen's question above asks whether the converse is true for finitely generated groups.
The aim of this work is to produce counterexamples to this question. More precisely, our main result is the following:

\begin{thm}\label{thm:main} There exists a finitely generated  group $L$ which has property (FA) and admits a non-trivial action on some $\R$-tree $T$, such that the arc stabilizers for this action are finite.
Moreover, $L$ is not a quotient of any finitely presented group with property (FA).
\end{thm}

The theorem of Sela \cite{Sela} mentioned above implies that a finitely generated group which acts non-trivially on an $\R$-tree
with trivial arc stabilizers cannot have (FA), and Guirardel's work \cite{Guirardel} shows that the same is true if one allows finite arc stabilizers of bounded size.
In Theorem \ref{thm:main} the stabilizers of a nested sequence of arcs will normally form a strictly increasing sequence of finite groups, in particular
the action of $L$ on $T$ is unstable. Nonetheless, our construction is sufficiently flexible and allows to ensure that the finite arc stabilizers have some extra properties. For example, one can take them to be
$p$-groups (see Theorem~\ref{thm:small_arc_stab} and the discussion above Lemma \ref{lem:K_j} in Section \ref{sec:M}).

The last claim of Theorem \ref{thm:main} can be compared with the fact that  any finitely generated group with property (F$\R$) \emph{is} a quotient of a finitely presented group
with this property (this follows from a theorem of Culler and Morgan \cite{Cul-Mor} establishing the compactness of the space of projective length functions for non-trivial actions of
any given finitely generated group on $\R$-trees; different proofs of this
fact, using ultralimits, were given by Gromov \cite{Gromov-random} and  Stalder \cite{Stalder}).

The pair $(L,T)$ from Theorem \ref{thm:main} is constructed as the limit of a strongly convergent sequence $(L_i,T_i)_{i \in \N}$, where each $L_i$ is a group splitting as a
free amalgamated product over a finite subgroup and $T_i$ is the Bass-Serre tree associated to this splitting. The morphism from $T_i$ to $T_{i+1}$ is not simplicial (but it is a morphism of $\R$-trees),
as it starts with edge subdivision and then applies a sequence of edge folds (see Section \ref{sec:folding_seq}).
We analyze this morphism carefully in order to control the arc stabilizers for the resulting action of $L$ on $T$. The construction of $L_i$ uses an auxiliary group $M$
satisfying certain properties (see \descref{(P1)}-\descref{(P4)} below). The main technical content of the paper is in Section \ref{sec:M}, where we construct a suitable group $M$ using small cancellation
theory over hyperbolic groups, and in Section \ref{sec:strong_conv}, where we prove that the corresponding sequence $(L_i,T_i)_{i \in \N}$ is strongly convergent (in the sense of Gillet and Shalen \cite{Gillet-Shalen}).

Theorem \ref{thm:main}  also shows that finite presentability is a necessary assumption in the result of Gillet and Shalen mentioned above (when the $\mathbb Q$-rank of $\Lambda$ is $1$), because the group $L$
can be seen to act non-trivially on a $\mathbb D$-tree, where $\mathbb D$ denotes the group of dyadic rationals -- see Remark \ref{rem:dyadic} below.

However, we start this paper with a short and elementary proof that finitely generated groups which have (FA) but do not have (F$\R$) exist,
a fact first proved in our preprint \cite{D-M} with M. Dunwoody, -- see Section \ref{sec:short}.
It is based on the idea that it is possible to avoid many technicalities required for the proof of Theorem \ref{thm:main} if one is ready to give up the control over the limit $\R$-tree.
Indeed, the amalgamated products $L_i$ and the epimorphisms $\phi_i:L_i \to L_{i+1}$, $i \in \N$, can be constructed in a purely algebraic way starting from
any finitely generated group $M$ with the first two properties \descref{(P1)}, \descref{(P2)} below, and \descref{(P3)} is enough to ensure that the direct limit
$L=\lim_{i\to \infty} (L_i,\phi_i)$ has property (FA). The fact that $L$ acts non-trivially on some $\R$-tree $T'$ can be established by using the general
`existence' result of Culler and Morgan \cite{Cul-Mor} mentioned above (in contrast, the construction of the $\R$-tree $T$ from Theorem \ref{thm:main} is quite explicit).
This also gives an extra benefit that the auxiliary group $M$ is easier to construct, as it does not have to satisfy the last property \descref{(P4)}
(which is needed to prove that the convergence of $(L_i,T_i)_{i \in \N}$ is strong).
The main disadvantage of this approach is that we have no control over the arc stabilizers for the action of $L$ on the $\R$-tree $T'$.

\smallskip
\noindent {\bf Acknowledgement.} The author is grateful to Gilbert Levitt and Yves de Cornulier for their comments on the first version of this article, and to the referee for
giving many useful remarks and suggestions.

\section{A short proof that (FA) does not imply (F$\R$)} \label{sec:short}
In this section we will present a short proof of a simplification of Theorem \ref{thm:main}, which gives no information about arc stabilizers:

\begin{thm}\label{thm:short} There exists a finitely generated  group $L$ which has property (FA) and admits a non-trivial action on some $\R$-tree, i.e., $L$ does not have (F$\R$). Moreover, $L$ is not a
quotient of any finitely presented group with property (FA).
\end{thm}

The first proof of Theorem \ref{thm:short} appeared in the preprint \cite{D-M} in 2012. Unfortunately this article was subsequently withdrawn from arXiv, due to
a disagreement between its authors and will not be published. The proof of Theorem \ref{thm:short} given in this section is inspired by the ideas from \cite{D-M}, which were
obtained in collaboration of the author with M. Dunwoody. However, the construction here is quite different from the one in \cite{D-M}, as it uses a single base group $M$ with property (FA) instead of a sequence of such groups, which allows to shorten the proof of the property (FA) for the limit group $L$. Moreover our argument below is purely group-theoretic and, unlike the proof from \cite{D-M}, it does not require any familiarity with the theory of tree foldings. Finally, the folding sequence employed in \cite{D-M}
cannot be used to produce a strong limit of simplicial trees, which is an essential ingredient in our proof of the main result (Theorem \ref{thm:main}).

\subsection{The groups $L_i$}\label{subsec:L_i}

Given a group  $G$, a subset $S \subseteq G$ and elements $g,h \in G$, throughout the paper we will employ the notation
$h^g:=ghg^{-1}$ and $S^g:=\{gsg^{-1} \mid s \in S\}$. We will also use $\N$ to denote the set of natural numbers $\{1,2,\dots\}$ (without zero).

The proof of Theorem \ref{thm:short} will make use of a finitely generated group $M$, containing
 a strictly ascending sequence of subgroups $G_0<G_1<G_2< \dots$ together with elements $a_i \in M$,
$i \in \N$,  such that the following two conditions are satisfied for all $i \in \N$:

\begin{itemize}
\descitem{(P1)} \normalfont $a_i$ centralizes $G_{i-1}$ in $M$;
\descitem{(P2)} \normalfont  $M= \langle G_i, G_i^{ a_i} \rangle$.
\end{itemize}

For each $i \in \N$, let $M_i$ be a copy of $M$ with a fixed isomorphism $\beta_i:M \to M_i$.
Let $L_i:=M \ast_{G_{i-1}=\beta_i(G_{i-1})} M_i$ be the amalgamated free product of $M$ and $M_i$, given by the following presentation:
\begin{equation}\label{eq:L_i-pres}
L_i= \langle M,M_i \mid g=\beta_i(g) \mbox{ for all } g \in G_{i-1} \rangle.
\end{equation}

\subsection{The epimorphism from $L_i$ to $L_{i+1}$}\label{subsec:homom}
The next lemma defines an epimorphism from $L_i$ to $L_{i+1}$ and lists some of its properties.

\begin{lemma} \label{lem:fold_props}
For each $i \in \N$ there is a unique  homomorphism $\phi_i: L_i \to L_{i+1}$ such that
\begin{equation}\label{eq:phi_i-def}
\phi_i(g)=g \quad\forall\, g \in M, \mbox{ and } \phi_i(h)=\beta_{i+1}(a_i)\beta_{i}^{-1}(h) \beta_{i+1}(a_i^{-1}) \quad\forall\, h \in M_i.
\end{equation}

Moreover, the homomorphism $\phi_i$ has the following properties:

\begin{itemize}
\item[(i)] the restrictions of $\phi_i$ to $M\leqslant L_i$ and to $M_i \leqslant L_i$ are injective, $\phi_i(M)=M \leqslant L_{i+1}$ and
$\phi_i(M_i)=M^{\beta_{i+1}(a_i)} \leqslant L_{i+1}$;
\item[(ii)]  $\phi_i:L_i \to L_{i+1}$ is surjective;
\item[(iii)] $L_{i+1}=\langle \phi_i(M_i),M_{i+1} \rangle$.
\end{itemize}
\end{lemma}

\begin{proof} By the universal property of amalgamated free products, to verify that the homomorphism $\phi_i$ satisfying \eqref{eq:phi_i-def} exists,
we just need to check that it is well-defined on the amalgamated subgroup $G_{i-1}=\beta_i(G_{i-1})$. So, suppose that $h=\beta_i(g) \in M_i$ for some $g \in G_{i-1}$.
Then, recalling that $g=\beta_{i+1}(g)$ in $L_{i+1}$ by definition and $g^{a_i}=g$ in $M$ by \descref{(P1)}, we get
$$\phi_i(h)=\beta_{i+1}(a_i)\beta_{i}^{-1}(h) \beta_{i+1}(a_i^{-1})=g^{\beta_{i+1}(a_i)}=\beta_{i+1} \left( g^{a_i}\right )=\beta_{i+1}(g)=g.$$
Thus $\phi_i(h)=g=\phi_i(g)$, as required. Evidently the homomorphism $\phi_i: L_i \to L_{i+1}$, satisfying \eqref{eq:phi_i-def}, is unique because $L_i$ is generated by $M$ and $M_i$.

Claim (i) follows immediately from the definition of $\phi_i$.
Now, the group $M_{i+1}$ is generated by $\beta_{i+1}(G_i)$ and $\beta_{i+1}(G_i^{a_i})$   by condition \descref{(P2)}, which implies that in $L_{i+1}$ one has
$$L_{i+1}=\langle M, \beta_{i+1}(G_i), \beta_{i+1}(G_i^{a_i}) \rangle= \langle M, M^{\beta_{i+1}(a_i)} \rangle =\langle \phi_i(M),\phi_i(M_i) \rangle = \phi_i(L_i),$$
yielding claim (ii). To prove claim (iii), notice that
$$L_{i+1}=\langle M,M_{i+1} \rangle= \langle M^{\beta_{i+1}(a_i)}, M_{i+1} \rangle =\langle \phi_i(M_i) ,M_{i+1} \rangle,$$
because $\beta_{i+1}(a_i) \in M_{i+1}$ and $\phi_i(M_i)=M^{\beta_{i+1}(a_i)}$ by claim (i).
\end{proof}

\begin{rem}
Since $G_{i-1} \leqslant G_i$, there is a `na\"\i ve' epimorphism $\kappa_i:L_i \to L_{i+1}$, which restricts to the identity map on $M$ and to the composition $\beta_{i+1}\circ \beta_i^{-1}:M_i \to M_{i+1}$ on $M_i$.
However, this is \emph{different} from the epimorphism $\phi_i:L_i \to L_{i+1}$ described above: for example, by claim (i) of Lemma  \ref{lem:fold_props},
$\phi_i$ sends both $M$ and $M_i$ to conjugates of $M$, while $\kappa_i(M_i)=M_{i+1}$. It is not difficult to see that these `na\"\i ve' epimorphisms are  actually useless for the purposes of this paper.
\end{rem}

\subsection{The limit group $L$ and property (FA)}\label{subsec:(FA)}
Let the sequence of groups $L_i$ and the epimorphisms $\phi_i:L_i \to L_{i+1}$, $i \in \N$, be as above.
For $1 \le i < j$, let $\phi_{ij}:L_i \to L_j$ denote the composition $\phi_{ij}:=\phi_{j-1} \circ\dots\circ \phi_i$ (thus $\phi_{i,i+1}=\phi_i$).
We also define $\phi_{ii}:L_i \to L_i$ to be the identity map.

The sequence of groups $L_i$, equipped  with the epimorphisms $\phi_{ij}$, forms a directed family which has a direct limit, denoted by $L$.
This means that for each $i \in \N$ there is an epimorphism $\psi_i:L_i \to L$ such that
\begin{equation}\label{eq:compos-phi}
\psi_{j} \circ \phi_{ij}=\psi_i \mbox{ whenever }1 \le i \le j .
\end{equation}

In this section we will show that $L$ has property (FA), provided the same holds for $M$.
So, assume that, in addition to \descref{(P1)} and \descref{(P2)}, the group $M$ satisfies
\begin{itemize}
\descitem{(P3)}  \normalfont $M$ has property (FA).
\end{itemize}

\begin{lemma}\label{lem:(FA)} Suppose that a finitely generated group $M$, a strictly ascending sequence of its subgroups
$G_0 < G_1 <  \dots$ and elements $a_1,a_2,\dots \in M$ satisfy conditions \descref{(P1)}-\descref{(P3)}.
Then the limit group $L$ defined above has property (FA).
\end{lemma}

\begin{proof} Assume that $L$ acts simplicially without edge inversions on a simplicial tree $S$.
Let $\overline{M}:=\psi_1(M) \leqslant L$ and $\overline{M}_i:=\psi_i(M_i) \leqslant L$, $i \in \N$. For any subgroup $H \leqslant L$, $\Fix(H)$ will denote the set of points in
$S$ fixed by all elements of $H$.

By \descref{(P3)}, there is a vertex $u \in \Fix(\overline{M})$, and for each $i \in \N$ the fixed point set
$\Phi_i:=\Fix(\overline{M}_i)$ is a non-empty subtree of $S$. Since the tree $S$ is simplicial, we can choose $i \in \N$ so that $\d_S(u, \Phi_i)$ is minimal, where $\d_S$ denotes the standard simplicial metric on $S$.

If $u \in \Phi_i$ then it is fixed by both $\overline{M}$ and $\overline{M}_i$. But $L$ is generated by these two subgroups, as $L_i =\langle M, M_i\rangle$ and so
$$L=\psi_i(L_i)=\langle \psi_i(M), \psi_i(M_i) \rangle = \left\langle \psi_i(\phi_{1i}(M)), \overline{M}_i \right\rangle =\left\langle \overline{M},\overline{M}_i \right\rangle,$$
where we used \eqref{eq:compos-phi} together with claim (i) of Lemma \ref{lem:fold_props}. Hence $u \in \Fix(L)$.

Thus, we can further assume that $\d_S(u, \Phi_i)$ is a positive integer. Let $v \in \Phi_i$ be the vertex closest to $u$ and choose any vertex $w \in \Phi_{i+1}$.
Clearly the geodesic segment $[u,w]$ is fixed by $\psi_{i+1}(G_i)=\psi_{i+1}(\beta_{i+1}(G_i))$, as $G_i=\beta_{i+1}(G_i)=M \cap M_{i+1}$ in $L_{i+1}$.
On the other hand, $\beta_{i+1}(G_i^{a_i})=G_i^{\beta_{i+1}(a_i)} \leqslant M_{i+1} \cap \phi_i(M_i)$ in $L_{i+1}$, which implies that the entire segment $[w,v]$ is fixed by the image of $\beta_{i+1}(G_i^{a_i})$ in $L$.

Since $S$ is a simplicial tree, the intersection of the geodesic segments $[u,v]$, $[u,w]$ and $[w,v]$ is a single vertex $x$ of $S$, which, by the above argument,
must be fixed by both $\psi_{i+1}(\beta_{i+1}(G_i))$ and
$\psi_{i+1}(\beta_{i+1}(G_i^{a_i}))$. But the latter two subgroups generate $\overline{M}_{i+1}=\psi_{i+1}(M_{i+1})$ by \descref{(P2)}.
Thus $x \in \Phi_{i+1}$. Recalling that  $x \in [u,v]$, the choice of $i$ and $v$ implies that $x=v$.

It follows that $v \in \Phi_{i} \cap \Phi_{i+1}$. From this we can conclude that $v \in \Fix(L)$, as $L$ is generated by $\overline{M}_i$ and $\overline{M}_{i+1}$. Indeed, the
latter can be derived from claim (iii) of Lemma~\ref{lem:fold_props} and \eqref{eq:compos-phi}, as
$$L=\psi_{i+1}(L_{i+1})= \psi_{i+1} \left(\langle \phi_i(M_i),M_{i+1}\rangle \right)=\left\langle \psi_i(M_i), \psi_{i+1}(M_{i+1})\right\rangle=\left\langle \overline{M}_i, \overline{M}_{i+1}\right\rangle.$$
Therefore we have shown that any simplicial action without edge inversions of $L$ on a simplicial tree $S$ has a global fixed point, which means that $L$ has property (FA).
\end{proof}

\subsection{Using Thompson's group $V$ as $M$}\label{subsec:V}  In this subsection we will explain that one can take $M$ to be  Thompson's group $V$. Recall (see \cite{CFP}) that $V$ is
the group of all piecewise linear right continuous self-bijections of the interval $[0,1)$, mapping dyadic rationals to themselves,
which are differentiable in all but finitely many dyadic rational numbers and such that at every interval, where the function is linear, its derivative is a power of $2$.

It is well-known that $V$ is finitely generated and even finitely presented \cite{CFP}. The fact that $V$ has property (FA) is  proved in \cite[Thm. 4.4]{Farley},
thus \descref{(P3)} holds for $M=V$. For each $i=0,1,2, \dots$, let
$G_i:=\St_V([0,1/2^{i+1}))$ be the pointwise stabilizer of the interval $[0,1/2^{i+1})$ in $V$, i.e.,
$$G_i=\{f \in V \mid f(x)=x \mbox{ whenever } x \in [0,1/2^{i+1}) \} \leqslant V.$$
Thus $G_0=\St_V ([0,1/2))$, $G_1=\St_V([0,1/4))$, etc.
Evidently $G_0<G_1<G_2< \dots$ in $V$. Finally, for each $i \in\N$, we pick the function $a_i:[0,1) \to [0,1)$  according to the formula
\begin{equation*}\displaystyle
a_i(x):= \left\{
\begin{array}{ll}
  x+\frac{1}{2^{i+1}} & \mbox{if } x \in \left[0,\frac{1}{2^{i+1}}\right) \\
 x-\frac{1}{2^{i+1}} & \mbox{if } x \in \left[\frac{1}{2^{i+1}},\frac{1}{2^i}\right) \\
  x & \mbox{if } x \in \left[\frac{1}{2^i},1\right)
\end{array} \right. ,
\end{equation*}
in other words, $a_i$ simply permutes the intervals $[0,1/2^{i+1})$ and $[1/2^{i+1},1/2^i)$. Clearly $a_i \in V$ and $a_i$ commutes with any element from $G_{i-1}=\St_V([0,1/2^i))$ in $V$, for all $i \in \N$.
Thus \descref{(P1)} is satisfied. Observe that $G_i=\St_V([0,1/2^{i+1}))$ and $G_i^{a_i}=\St_V([1/2^{i+1},1/2^i))$ in $V$, so to verify \descref{(P2)} it is enough to show that $V$ is generated by
$\St_V([0,1/2^{i+1}))$ and $\St_V([1/2^{i+1},1/2^i))$ for all $i \in \N$.
The latter is a straightforward exercise. Indeed, recall that $V$ is generated by four elements $A,B,C$ and $\pi_0$ (see \cite{CFP}), defined as follows:

\begin{equation*}
A(x)=\left\{ \begin{array}{ll}
         \frac{x}2 & \mbox{if } x \in \left[0,\frac{1}{2}\right)\\
         x-\frac{1}{4} & \mbox{if } x \in \left[\frac12,\frac{3}{4}\right)\\
         2x-1 &   \mbox{if }x \in \left[\frac34,1\right)
         \end{array} \right. , \qquad
B(x)=\left\{ \begin{array}{ll}
         x & \mbox{if } x \in \left[0,\frac12\right)\\
         \frac{x}2+\frac{1}{4} & \mbox{if }x \in \left[\frac12,\frac{3}{4}\right)\\
         x-\frac{1}{8} & \mbox{if } x \in \left[\frac{3}{4},\frac{7}{8}\right)\\
         2x-1 & \mbox{if } x \in \left[\frac78,1\right)
         \end{array} \right. ,
\end{equation*}

\begin{equation*}
C(x)=\left\{ \begin{array}{ll}
         \frac{x}2+\frac{3}{4} & \mbox{if } x \in \left[0,\frac12\right)\\
         2x-1 & \mbox{if } x \in \left[\frac12,\frac{3}{4}\right)\\
         x-\frac{1}{4} & \mbox{if } x \in \left[\frac34,1\right)
         \end{array} \right. , \qquad
\pi_0(x)=\left\{ \begin{array}{ll}
         \frac{x}2+\frac{1}{2} & \mbox{if } x \in \left[0,\frac12\right)\\
         2x-1 & \mbox{if } x \in \left[\frac12,\frac{3}{4}\right)\\
         x & \mbox{if } x \in \left[\frac34,1\right)
         \end{array} \right. .
\end{equation*}

Fix any $i \in \N$.
One immediately notices that $B \in \St_V([0,\frac12)) \leqslant G_i$, and $\pi_0 \in \St_V([\frac34,1))$.
Clearly there exists an element $D_1 \in \St_V\left([0,\frac{1}{2^{i+1}})\right)=G_i$ such that $D_1\left([\frac34,1)\right)=[\frac{1}{2^{i+1}},\frac{1}{2^{i}})$,
hence $\St_V\left([\frac34,1)\right)=D_1^{-1} \, \St_V\left([\frac1{2^{i+1}},\frac1{2^i})\right)  D_1 =D_1 ^{-1}G_i^{a_i} D_1$. Therefore $B,\pi_0 \in \langle G_i,G_i^{a_i} \rangle$.

We can also observe that $B^{-1}A \in \St_V\left([\frac78,1)\right)$ and $\pi_0^{-1}C \in \St_V\left([\frac12,\frac34)\right)$. So, arguing as above we can find elements
$D_2,D_3 \in \St_V\left([0,\frac{1}{2^{i+1}})\right)=G_i$ such that $B^{-1} A \in D_2^{-1} G_i^{a_i} D_2 $ and $\pi_0^{-1} C \in D_3^{-1} G_i^{a_i} D_3$, which yields
that $A,C \in \langle G_i,G_i^{a_i} \rangle$. Thus $V= \langle G_i,G_i^{a_i} \rangle$  for any $i \in \N$, as claimed.
Hence the group $M=V$ satisfies properties \descref{(P1)}-\descref{(P3)} above.

\subsection{Proof of the weaker theorem}
\begin{proof}[Proof of Theorem \ref{thm:short}] Let $M$, the sequence of its subgroups $G_0<G_1< \dots$, and the elements $a_i \in M$, $i \in \N$, be as in Subsection \ref{subsec:V}.
Then we can define the groups $L_i$ and the homomorphisms $\phi_i:L_i \to L_{i+1}$ as above, and we will let $L$ be the direct limit of the sequence $(L_i,\phi_i)_{i \in \N}$.
It follows that there is a epimorphism $\psi_1: L_1  \to L$, implying, in particular,  that $L$ is finitely generated. Moreover, $L$ has property (FA) by Lemma \ref{lem:(FA)}.

Recall that, by definition, each $L_i$ splits non-trivially as an amalgamated free product, hence it admits a non-trivial action on the associated simplicial Bass-Serre tree $T_i$ (cf. \cite[I.4.1, Thm. 7]{Serre}),
$i \in \N$. In particular, $L_i$ does not have property (F$\R$) for any $i \in \N$. Since property (F$\R$) is open in the topology of marked groups (see \cite[Thm. 4.7]{Stalder} or \cite[Sec. 3.8.B]{Gromov-random}),
its complement is closed, and so the group $L$ also does not have (F$\R$), as a direct limit of groups without this property (because direct limits are limits in the topology of marked groups).

An alternative way to prove that $L$ has a non-trivial action on some $\R$-tree would be to use an earlier result of Culler and Morgan \cite{Cul-Mor} about compactness of the space of non-trivial
projective length functions for actions of a finitely generated group on $\R$-trees. Indeed, since $L_i$ is an epimorphic image of $L_1$, for each $i \in \N$ we get a non-trivial action of $L_1$ on $T_i$
(which factors through the action of $L_i$).
The set of such actions determines a sequence in the space ${\rm PLF}(L_1)$, of \emph{non-trivial projective length functions} of $L_1$ on $\R$-trees -- see \cite{Cul-Mor}. In \cite[Thm.~4.5]{Cul-Mor}
it is shown that the space ${\rm PLF}(L_1)$, equipped with a natural topology, is compact, which implies that the above sequence has a subsequence converging to a non-trivial (projective) length function $\lambda:L_1 \to \R$.
It is easy to see that $\lambda$ determines a non-trivial (projective) length function $\overline\lambda:L \to \R$ of $L$ (defined by $\overline\lambda(\psi_1(g)):=\lambda(g)$ for any $g \in L_1$),
yielding a non-trivial $L$-action on some $\R$-tree.

The final assertion of the theorem, is a consequence of a standard argument, showing that every finitely presented group $P$ which maps onto the direct limit $L$ must actually map onto some $L_i$
(see \cite[Lemma 3.1]{Cor-Kar}). Hence $P$ will act non-trivially on the simplicial tree $T_i$, and so it does not have (FA).
\end{proof}

\section{Preliminaries}
The rest of this paper is devoted to proving Theorem \ref{thm:main}. In this section we will recall some theory and terminology that will be used later on.

\subsection{Notation}
If $G$ is a group acting on a set $X$ and $Y \subset X$, then $\St_{G}(Y) \leqslant G$ will denote the \emph{pointwise stabilizer} of $Y$ in $G$.
If $e$ is an edge in a simplicial tree $S$, then $e_-$ and $e_+$ will denote the two endpoints of $e$ in $S$.

\subsection{Lambda-trees} 
Let $\Lambda$ be an ordered abelian group. A set $X$, equipped with a function $\d:X \times X \to \Lambda$, is a \emph{$\Lambda$-metric space}, if $\d$ enjoys the standard axioms of a metric
(it is positive definite, symmetric and satisfies the triangle inequality). In this paper the group $\Lambda$ will always be a subgroup of $\R$ (under addition).

Given a $\Lambda$-metric space $(X,\d)$,  a \emph{geodesic segment} in $X$ is a subset isometric to an interval
$[\lambda,\mu]_\Lambda:=\{\upsilon \in \Lambda \mid \lambda \le \upsilon \le \mu\}$ for some $\lambda,\mu \in \Lambda$, $\lambda \le \mu$.
A geodesic segment  in $X$ is an \emph{arc} if it is not degenerate (i.e., its endpoints are distinct).

$(X,d)$ is \emph{geodesic} if for any two points $x,y \in X$ there exists a geodesic segment $[x,y]$ joining them.
Intuitively, a geodesic $\Lambda$-metric space is a $\Lambda$-tree if it does not contain non-trivial simple loops.
Formally, $(X,\d)$ is a \emph{$\Lambda$-tree} if it is geodesic, the intersection of any two geodesic segments with a common endpoint is a geodesic segment in $X$,
and the union of any two geodesic segments which only share a single endpoint is a geodesic segment (see \cite{Chiswell}).

Standard examples of $\Lambda$-trees are $\Z$-trees (which are in one-to-one correspondence with simplicial trees) and
$\R$-trees (which can be characterized as connected metric spaces that are $0$-hyperbolic in the sense of Gromov -- see \cite[Lemma 4.13]{Chiswell}).
Given a positive number $r \in \R$, any simplicial tree $S$ can be made into an $\R$-tree by proclaiming that every edge is isometric to the segment $[0,r]$
(the vertex set of $S$ then becomes an $\langle r \rangle$-tree, where $\langle r \rangle$ denotes the cyclic subgroup of $\R$ generated by $r$). $\R$-trees that can be obtained this way are called \emph{simplicial $\R$-trees}.
If $r=1$ the $\R$-tree obtained from $S$ is actually the standard geometric  realization of $S$. However, further on we will also be using $r=1/2^i$ for some $i \in \N$.

As explained in the Introduction the $\R$-tree $T$, on which the limit group $L$ acts non-trivially, will be constructed as a limit of some simplicial $\R$-trees $T_i$.
However the morphism from $T_i$ to $T_{i+1}$ will not be simplicial, as we perform edge subdivision. Therefore, we will use a more general notion of a morphism, suggested by Gillet and Shalen in \cite{Gillet-Shalen}.
Given two $\Lambda$-trees $S'$ and $S''$, a map $f:S' \to S''$ is a \emph{morphism} if for any two points $a,b \in S'$
there are points $a=x_0,x_1,\dots,x_n=b$ such that the geodesic segment $[a,b]$ is subdivided into the union of geodesic segments $[x_0,x_1] \cup \dots \cup [x_{n-1},x_n]$
and the restriction of $f$ to $[x_{i-1},x_i]$ is an isometric embedding of $\Lambda$-metric spaces, for every $i=1,\dots,n$ (see \cite[Sec. 1.7]{Gillet-Shalen}). This notion of a morphism allows to
subdivide edges and fold edges together, which is what we will employ later.

Since we will be interested in (isometric) group actions on trees, it is convenient to operate in the \emph{category of $\Lambda$-trees with symmetry}, which was also introduced in \cite{Gillet-Shalen}. The objects in
this category are pairs $(H,S)$, where $S$ is a $\Lambda$-tree and $H$ is a group with a fixed action on $S$ by isometries
(in the case when $S$ is a simplicial tree, we will also require that the action is simplicial and without edge inversions). Given two objects $(H',S')$ and $(H'',S'')$ in the
category of $\Lambda$-trees with symmetry, a \emph{morphism} between these objects is a pair $(\phi, \varphi)$, where $\phi:H' \to H''$ is a group homomorphism and $\varphi: S' \to S''$ is a morphism of $\Lambda$-trees
which is equivariant with respect to $\phi$, i.e.,  $\phi(h)\circ \varphi(s)=\varphi(h \circ s)$ for all $h \in H'$ and $s \in S'$.

A natural source of morphisms in the category of simplicial trees with symmetry comes from \emph{morphisms of graphs of groups}, which were introduced and
studied by Bass in \cite{Bass-covering}. Given two finite graphs of groups $\G'$ and $\G''$, a morphism $\G' \to \G''$ consists of a simplicial map between the underlying simplicial graphs together
with the collection of homomorphisms between the vertex and edge groups
of $\G'$ and (possibly conjugates of) the vertex and edge groups of $\G''$, satisfying natural compatibility conditions. We refer the reader to \cite[Sec.~2]{Bass-covering} for a
formal definition. Let $H'$, $H''$ be the fundamental groups and let $T'$, $T''$ be the associated Bass-Serre trees of $\G'$, $\G''$ respectively. In \cite[Prop. 2.4]{Bass-covering} Bass proves that
any morphism from $\G'$ to $\G''$ gives a homomorphism $\phi:H' \to H''$ and a morphism of simplicial trees $\f:T' \to T''$, which is equivariant with respect to $\phi$.
Clearly scaling the simplicial metrics on $T'$ and $T''$ by the same real number $r>0$ does not affect these maps, so if one views $T'$ and $T''$ as simplicial $\R$-trees, then
$(\phi,\f)$ becomes a morphism from $(H',T')$ to $(H'',T'')$ in the category of $\R$-trees with symmetry.

\subsection{Strong limits}\label{subsec:strong_lim}
Suppose that we are given a sequence $(T_i)_{i \in \N}$ of $\Lambda$-trees together with $\Lambda$-tree morphisms $\f_i:T_i \to T_{i+1}$, $i \in \N$. Then we can form a direct system
$(T_i,\f_{ij})$ of $\Lambda$-trees, by setting $\f_{ij}:=\f_{j-1} \circ \dots \circ\f_i: T_i \to T_j$, whenever $ 1 \le i<j$.

Let $\d_i$ denote the ($\Lambda$-) metric on $T_i$, $i \in \N$.
Following \cite[Sec. 1.20]{Gillet-Shalen}, we will say that the sequence $(T_i,\d_i,\f_{i})_{i \in \N}$ \emph{converges strongly} if for any $l \in \N$ and any two points $x,y$ of $T_l$ there exists $k \in \N$ such that
$\d_j(\f_{lj}(s),\f_{lj}(t))=\d_k(\f_{lk}(s),\f_{lk}(t))$ for any $s,t \in [x,y]$
and all $j \ge k$. In particular, this implies that for all $x,y \in T_l$ the sequence of distances $\d_j(\f_{lj}(x),\f_{lj}(y)) \in \Lambda$, $j \ge l$,
eventually stabilizes (since $\Lambda$-tree morphisms are always distance-decreasing the latter condition is actually sufficient for the sequence to converge strongly).

Assuming that each map $\f_i:T_i \to T_{i+1}$ is surjective and the sequence $(T_i,\f_{i})_{i\in\N}$ converges strongly, one can construct the limit $\Lambda$-metric space $(T,\d)$ for this sequence
as follows (see \cite[Sec. 1.21]{Gillet-Shalen}). Define the pseudometric $\hat\d$ on $T_1$ by $\hat\d(x,y):=\lim_{i \to \infty}\d_i\bigl(\f_{1i}(x),\f_{1i}(y)\bigr)$ for all $x,y \in T_1$.
We now set $T$ to be the quotient of $T_1$ by the equivalence relation $\sim$, where $x \sim y$ if and only if $\hat\d(x,y)=0$.

In  \cite[Prop. 1.22 and 1.27]{Gillet-Shalen} Gillet and Shalen proved that  the function $\d:T \times T \to \Lambda$, given by
$\d(\bar x,\bar y):=\hat\d(x,y)$ for any choice $x, y \in T_1$ representing the equivalence classes $\bar x,\bar y \in T$, is a $\Lambda$-metric on $T$ and $(T,\d)$ is a $\Lambda$-tree.
In the case when $\Lambda=\R$ this can be easily shown using the $0$-hyperbolicity characterization, mentioned above. We will say that $(T,\d)$ is the \emph{limit $\Lambda$-tree} for the sequence  $(T_i,\d_i,\f_i)_{i \in \N}$.

For every $i \in \N$ we have a natural map of metric spaces $\theta_i:(T_i,\d_i) \to (T,\d)$ defined as follows. The map $\theta_1$ sends $x \in T_1$ to its equivalence class under $\sim$.
And if $i>1$ then for any $y \in T_i$,
choose an arbitrary $x \in T_1$ such that $y=\f_{1i}(x)$ and set $\theta_i(y):=\theta_1(x)$ (this gives a well-defined map since for any other point $x' \in T_1$,
with $\f_{1i}(x')=y$, one has $x \sim x'$). In \cite[Prop. 1.22]{Gillet-Shalen} it is shown that these maps $\theta_i:T_i \to T$ are actually morphisms of $\Lambda$-trees.

\section{Construction of the morphisms}\label{sec:folding_seq}
The desired pair $(L,T)$ from Theorem \ref{thm:main} will be constructed as a direct limit of a sequence $(L_i,T_i)_{i \in \N}$, where $L_i$ is a group acting non-trivially by isometries (and without inversions)
on a simplicial $\R$-tree $T_i$ in the category of $\R$-trees with symmetry. In fact, as we will see later, the groups $L_i$ will be the amalgams from Subsection \ref{subsec:L_i}, for suitable choice of the group $M$,
and $T_i$ will be the corresponding Bass-Serre trees.
In order to show that for each $i \in \N$ there is a natural morphism between the pairs $(L_i,T_i)$ and $(L_{i+1},T_{i+1})$, we will look at the corresponding graphs of groups.
Namely, we will construct a sequence of graphs of groups $\G_i$ so that $L_i$ will be the fundamental group of $\G_i$ and $T_i$ will be the geometric realization of the
corresponding Bass-Serre tree (where the standard simplicial metric is appropriately rescaled).

As before, we will need an auxiliary  finitely generated group $M$ which contains a strictly ascending sequence of subgroups $G_0<G_1<G_2< \dots$ together with elements $a_i \in M$,
$i \in \N$,  satisfying properties \descref{(P1)} and \descref{(P2)} from Subsection \ref{subsec:L_i}.
Again, for each $i \in \N$, we take a copy $M_i$ of $M$, and fix an isomorphism $\beta_i:M \to M_i$.

Let $\G_i$ be the graph of groups with one edge, where the two vertex groups are $M$ and $M_i$ and the edge group
is $G_{i-1}$, equipped with the natural inclusion into $M$, so that the embedding of this edge group into $M_i$ is given by the restriction of $\beta_i$ to $G_{i-1}$ (see the first line of Figure \ref{fig:folding_seq}).
Let $L_i$ be the fundamental group of $\G_i$ and let $T_i$ be the corresponding Bass-Serre tree.
Then $L_i$ is naturally isomorphic to the amalgamated free product $M\ast_{G_{i-1}=\beta_i(G_{i-1})} M_i$ with presentation \eqref{eq:L_i-pres}, which was discussed in Subsection \ref{subsec:L_i}.
Each tree $T_i$ will be viewed as a simplicial $\R$-tree, equipped with a natural metric $\d_i$ in which every edge is isometric to the interval
$[1,1/2^{i-1}]$ (i.e., the standard simplicial metric of $T_i$ is downscaled by $2^{i-1}$).

Now, let us describe the morphism $(\phi_i,\f_i): (L_i,T_i) \to (L_{i+1},T_{i+1})$ in the category of $\R$-trees with symmetry. This morphism is obtained via a composition of several intermediate morphisms, which we call steps.
The pictorial illustration of these steps, in terms of the respective graphs of groups, is given in Figure \ref{fig:folding_seq}. The morphism from the first step simply corresponds to edge subdivision in $T_i$.
The intermediate morphisms from the remaining steps will come from the morphisms between the corresponding graphs of groups (in the sense of Bass \cite{Bass-covering}).

\begin{figure}[ht]\begin{center}
\input{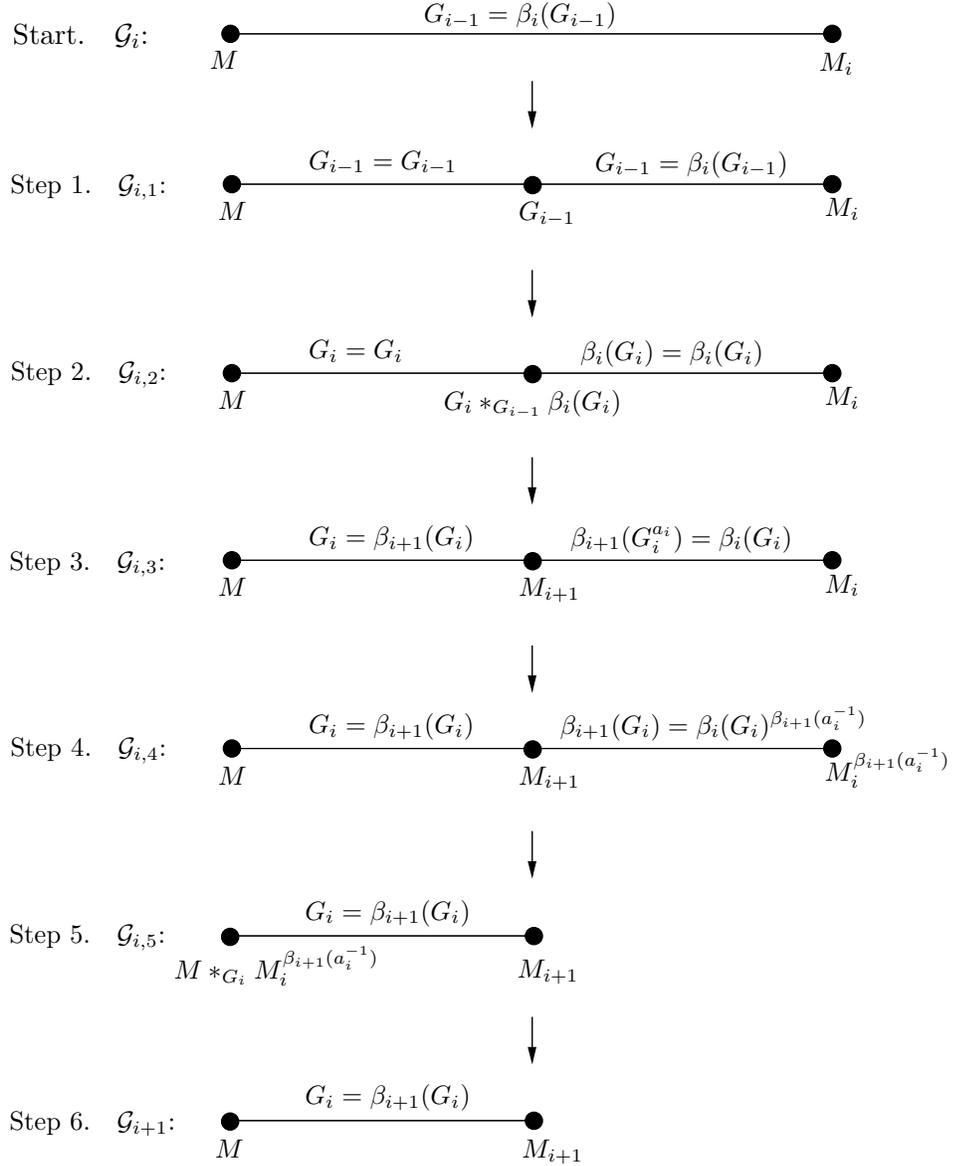}
\caption{\label{fig:folding_seq} The morphism from $\mathcal{G}_i$ to $\mathcal{G}_{i+1}$.}
\end{center}
\end{figure}

\medskip\noindent {\bf Step 1.} We start by inserting a new vertex with the group $G_{i-1}$ at the middle of the edge in $\G_i$ to obtain the graph of groups $\G_{i,1}$.
This means that the corresponding Bass-Serre tree $T_{i,1}$ is obtained from $T_i$
by subdividing all edges. Evidently the fundamental group $L_{i,1}$ of $\G_{i,1}$ is the same as before, i.e., it is equal to $L_i$. Strictly speaking, this does not give rise to a graph of groups morphism
from $\G_i$ to $\G_{i,1}$, as the induced map on the underlying graphs is not simplicial. However, clearly we do have a morphism $(L_i,T_i) \to (L_{i,1}, T_{i,1})$ in the category of simplicial $\R$-trees with symmetry,
where the edge length in $T_{i,1}$ is defined to be half of the edge length in $T_i$.

\medskip\noindent {\bf Step 2.} Clearly, the subgroup of $L_{i,1}$ generated by $G_i$ and $\beta_i(G_i)$ is isomorphic to the free amalgamated product $G_i\ast_{G_{i-1}=\beta_i(G_{i-1})} \beta_i(G_i)$.
To pass from $\G_{i,1}$ to $\G_{i,2}$, we apply a graph of groups morphism, which does not change the underlying graph,
sends the vertex groups $M$ and $M_i$ to themselves (identically) and naturally embeds the middle vertex group $G_i$ into the subgroup $\langle G_i,\beta_i(G_i) \rangle \leqslant L_i$.
It also sends the edge groups to the corresponding edge groups using the natural inclusions $G_{i-1} \hookrightarrow G_i$ and
$\beta_i(G_{i-1}) \hookrightarrow \beta_i(G_i)$.

It is not difficult to see that the Bass-Serre tree $T_{i,2}$ for $\G_{i,2}$ is obtained from $T_{i,1}$ by folding some edges together. In fact,
if the group $G_i$ is finitely generated, this morphism can be obtained as a composition of several Type IIA folds in the terminology of
Bestvina and Feighn -- see \cite[Sec. 2]{BF-complexity}. Recall (cf. \cite[Sec. I.4.1 and I.5.3]{Serre}) that vertices of the Bass-Serre tree $T_{i,1}$
correspond to left cosets of $M$, $M_i$ and $G_{i-1}$ and edges correspond to left cosets of $G_{i-1}$.
In these terms, the morphism from $T_{i,1}$ to $T_{i,2}$ can be described as follows: if $v$ is a vertex of $T_{i,1}$ corresponding to the coset $xM$, for some $x \in L_{i,1}$,
and $e$ is an edge adjacent to $v$, corresponding to the coset of $xG_{i-1}$, then $e$ is folded together with all the edges
adjacent to $v$ which come from the coset $xG_i$ of $G_i$
(such edges will correspond to the cosets of the form $xyG_{i-1}$, where $y$ runs over representatives of the cosets $G_i/G_{i-1}$);
similarly, edges adjacent to a vertex corresponding to a coset of $M_{i}$ are folded with the edges at that vertex corresponding to the same coset of
$\beta_i(G_i)$.

As before, the fundamental group $L_{i,2}$ of $\G_{i,2}$ is unchanged, i.e., it is naturally isomorphic to $L_i$ (the standard presentation of $L_{i,2}$ can be obtained from the presentation of $L_i$ by
applying a finite number of Tietze transformations -- see \cite[Sec. II.2]{L-S}).

\medskip\noindent {\bf Step 3.} The graph of groups $\G_{i,3}$ is obtained from $\G_{i,2}$ by applying a \emph{vertex morphism} (using the terminology of \cite{Dun-folding_seq}). The underlying graph stays the same and the maps
between the corresponding vertex and edge groups are natural isomorphisms/identities, except for the vertex groups in the middle, where the epimorphism
$$ G_i\ast_{G_{i-1}=\beta_i(G_{i-1})} \beta_i(G_i) \to M_{i+1}, $$ sends
$G_i$ to $\beta_{i+1}(G_i)\leqslant M_{i+1}$ (via the map $g \mapsto \beta_{i+1}(g)$ for all $g \in G_i$) and $\beta_i(G_i)$ to $\beta_{i+1}(G_i^{a_i})\leqslant M_{i+1}$
(via the map $\beta_i(g) \mapsto \beta_{i+1}(g^{a_i})$ for all $g \in G_i$).
Note that we used \descref{(P1)} together with the universal property of the amalgamated free products to conclude that these maps extend to a homomorphism between the middle vertex groups of $\G_{i,2}$ and $\G_{i,3}$.
The fact that this homomorphism is surjective follows from condition \descref{(P2)} above, as $M_{i+1}=\langle \beta_{i+1}(G_i),\beta_{i+1}(G_i^{a_i})\rangle$.

\medskip\noindent {\bf Step 4.} In this step, we keep the same group $L_{i,4}=L_{i,3}$ with the same action on the same tree $T_{i,4}=T_{i,3}$, but we choose a different fundamental domain for this action,
giving rise to the graph of groups $\G_{i,4}$. Again, this gives a graphs of groups morphism from $\G_{i,3}$ to $\G_{i,4}$, sending $M_i$ and the adjacent edge group $\beta_i(G_i)$ in $\G_{i,3}$ to the
conjugates of  $M_i^{\beta_{i+1}(a_i^{-1})}$ and the adjacent edge group $\beta_i(G_i)^{\beta_{i+1}(a_i^{-1})}$ in $\G_{i,4}$ by the element $\beta_{i+1}(a_i)$, which belongs to the vertex group
$M_{i+1}$ at the middle of $\G_{i,4}$. This step is only auxiliary, as it neither changes the group nor the tree on which it acts,
but it makes the description of the next step easier.

\medskip\noindent {\bf Step 5.} The graph of groups $\G_{i,5}$ consists of a single edge, where the `right' vertex group is $M_{i+1}$ and the `left' vertex group is the subgroup of $L_{i,4}$ generated by $M$ and
$M_i^{\beta_{i+1}(a_i^{-1})}$.
The morphism from $\G_{i,4}$ to $\G_{i,5}$ glues together  the two edges of the former.
The middle vertex group $M_{i+1}$ of $\G_{i,4}$ is mapped identically to the `right' vertex of $\G_{i,5}$.
The maps from the vertex groups $M$ and $M_i^{\beta_{i+1}(a_i^{-1})}$ of $\G_{i,4}$ to the `left' vertex group of $\G_{i,5}$ are the natural inclusions.
On the level of the Bass-Serre trees, $T_{i,5}$ is obtained from $T_{i,4}$ by applying an edge folding of Type IA
(see \cite[Sec. 2]{BF-complexity}). The fundamental group $L_{i,5}$, of $\G_{i,5}$ is not affected and coincides with $L_{i,4}$ (as before, one can verify this by applying Tietze transformations to the standard presentation of $L_{i,4}$). It is also important to note, that for any edge $e$ of $T_{i,4}$, it is only folded with
edges that have the same stabilizer, therefore the stabilizer of $e$ in $L_{i,4}$ is mapped identically to
the stabilizer of its image in $L_{i,5}$.

Now we need to observe that the left vertex group $\langle M, M_i^{\beta_{i+1}(a_i^{-1})}\rangle \leqslant L_{i,4}$ in $\G_{i,5}$ is naturally isomorphic to the amalgamated free product
$$M \ast_{G_i=\beta_i(G_i)^{\beta_{i+1}(a_i^{-1})}} M_i^{\beta_{i+1}(a_i^{-1})} =\langle M,  M_i^{\beta_{i+1}(a_i^{-1})} \mid g=\beta_i(g)^{\beta_{i+1}(a_i^{-1})}, \mbox{ for all } g \in G_i\rangle.$$
Indeed, this can be seen by looking at Step 4 on Figure \ref{fig:folding_seq}, which shows that
$L_{i,4}$ has the presentation $$\langle M,M_i^{\beta_{i+1}(a_i^{-1})},M_{i+1} \mid g=\beta_{i+1}(g)=\beta_i(g)^{\beta_{i+1}(a_i^{-1})}, \mbox{ for all } g \in G_i\rangle, $$
which is also a presentation of the double amalgamated free product:
\begin{equation}\label{eq:double_amalg}
\left( M \ast_{G_i=\beta_i(G_i)^{\beta_{i+1}(a_i^{-1})}} M_i^{\beta_{i+1}(a_i^{-1})} \right) \ast_{G_i=\beta_{i+1}(G_i)} M_{i+1}.
\end{equation}
Therefore $L_{i,4}$ is naturally isomorphic to the double amalgamated free product \eqref{eq:double_amalg}, implying that  the subgroup
generated by $M$ and $M_i^{\beta_{i+1}(a_i^{-1})}$ is naturally isomorphic to their free amalgam along $G_i=\beta_i(G_i)^{\beta_{i+1}(a_i^{-1})}$.

\medskip\noindent {\bf Step 6.} To perform the final step, observe that the `left' vertex group in $\G_{i,5}$ is
isomorphic to the double $M \ast_{G_i=G_i} M$, of $M$ along $G_i$. Therefore, this double retracts onto $M$ by identifying the second copy of $M$ with the first one. More precisely,
the map $$\eta_i: M \ast_{G_i=\beta_i(G_i)^{\beta_{i+1}(a_i^{-1})}} M_i^{\beta_{i+1}(a_i^{-1})} \to M$$ is defined by
\begin{equation}\label{eq:eta_i-1}
\eta_i(g)=g \mbox{ for all } g \in M, \mbox { and }
\end{equation}

\begin{equation}\label{eq:eta_i-2}
\eta_i\left( h^{\beta_{i+1}(a_i^{-1})} \right)=\beta_i^{-1}(h) \mbox{ for all } h \in M_i.
\end{equation}

To show that these maps indeed can be combined to the homomorphism from the amalgamated free product to $M$,
one has to check that the formulas \eqref{eq:eta_i-1} and \eqref{eq:eta_i-2} give the same result for any $g \in G_i$.
Indeed if $g \in G_i$ then $g=\beta_i(g)^{\beta_{i+1}(a_i^{-1})}$ in $L_{i,5}$, so, using \eqref{eq:eta_i-2}, one gets
$$\eta_i(g)=\eta_i\left( \beta_i(g)^{\beta_{i+1}(a_i^{-1})}\right)= \beta_i^{-1} \left( \beta_i(g) \right)=g, $$
which  agrees with \eqref{eq:eta_i-1}.

The above epimorphism from the `left' vertex of $\G_{i,5}$ to the `left' vertex of $\G_{i+1}$ allows to apply the corresponding vertex morphism to the graph of groups $\G_{i,5}$,
resulting in the graph of groups $\G_{i+1}$. For the `right' vertex groups and for the edge groups in these graphs of groups the corresponding maps are the natural identifications/isomorphisms.
Let $\tilde\eta_i: L_{i,5} \to L_{i+1}$ denote the induced map of the fundamental groups. Then the restriction of $\tilde\eta_i$ to $M_{i+1}$ is the identity map and its restriction to the `left' vertex group is $\eta_i$.
Therefore, as $\beta_{i+1}(a_i) \in M_{i+1}$, in $L_{i+1}$ we have
\begin{equation}\label{eq:eta_i-3}
\tilde\eta_i(h)= \beta_{i+1}(a_i) \eta_i\left(h^{\beta_{i+1}(a_i^{-1})}\right)\beta_{i+1}(a_i^{-1})=\beta_i^{-1}(h)^{\beta_{i+1}(a_i)} \mbox{ for all } h \in M_i.
\end{equation}

\medskip
Thus we have constructed a sequence of morphisms (in the category of $\R$-trees with symmetry), starting with the pair $(L_i,T_i)$ and ending with the pair $(L_{i+1},T_{i+1})$.
Let $(\phi_i,\f_i): (L_i,T_i) \to (L_{i+1},T_{i+1})$ be the composition of these morphisms. Evidently $\phi_i$ restricts to the identity map on $M$, and \eqref{eq:eta_i-3} shows that it maps each $h \in M_i$ to
$\beta_i^{-1}(h)^{\beta_{i+1}(a_i)}$. Therefore $\phi_i:L_i \to L_{i+1}$ obtained this way is the same as the epimorphism from Lemma~\ref{lem:fold_props} in Subsection \ref{subsec:homom}.
In particular, further we can use all the claims of Lemma~\ref{lem:fold_props}.

Let us summarize the main properties of the morphism $(\phi_i,\f_i): (L_i,T_i) \to (L_{i+1},T_{i+1})$ which will be used later:

\begin{lemma}\label{lem:morph_summary} For any $i \in \N$, let $x$ and  $y$ be the vertices of the Bass-Serre tree $T_i$ corresponding to the subgroups $M$ and $M_i$ of $L_i$
respectively, and let $e$ be the edge of $T_i$, joining these vertices and corresponding to $G_{i-1} \leqslant L_i$. Let $e=e_1 \cup e_2$ be the subdivision of $e$
in the union of two segments $e_1$ and $e_2$ such that ${e_1}_-=x$, ${e_2}_+=y$ and ${e_1}_+={e_2}_-$ is the midpoint of $e$.
Then $\be_1:= \f_i(e_1)$ and $\be_2:=\f_i(e_2)$ are edges of $T_{i+1}$ meeting at the vertex $v$, which is the $\f_i$-image of the midpoint of $e$ in $T_{i+1}$, and the following
properties hold:
\begin{itemize}
  \item[(1)] $\St_{L_{i+1}}(\f_i(x))= M$, $\St_{L_{i+1}}(\f_i(y))=M^{\beta_{i+1}(a_i)}$ and $\St_{L_{i+1}}(v)=M_{i+1}$.
   \item[(2)] $\St_{L_{i+1}}(\be_1)=G_i$, $\St_{L_{i+1}}(\be_2)=\beta_{i+1}(G_i^{a_i})$; in particular, $\be_1\neq \be_2$ in $T_{i+1}$.
  \item[(3)] If $c \in M\setminus G_{i-1}$ then $e_1$ is identified with $c\circ e_1$ in $T_{i+1}$ if and only if $c \in G_i$.
  \item[(4)] If $c \in M_i\setminus \beta_i(G_{i-1})$ then $e_2$ is identified with $c\circ e_2$ in $T_{i+1}$ if and only if $c \in \beta_i(G_i)$.
\end{itemize}

\end{lemma}

\begin{proof} The fact that $\be_1$ and $\be_2$ are edges of $T_{i+1}$ is clear from the construction, and property (1) holds by Lemma \ref{lem:fold_props}.
The stabilizers of the images of $e_1$ and $e_2$ in $T_{i,2}$ increase to $G_i$ and $\beta_i(G_i)$ respectively at Step 2, but the remaining steps induce
isomorphisms on the edge stabilizers, so, according to Lemma \ref{lem:fold_props}, in $L_{i+1}$ we have
$$\St_{L_{i+1}}(\be_1)=\phi_i(G_i)=G_i=\beta_{i+1}(G_i) \mbox{ and } \St_{L_{i+1}}(\be_2)=\phi_i(\beta_i(G_i))=G_i^{\beta_{i+1}(a_i)}=\beta_{i+1}(G_i^{a_i}).$$
Therefore
\begin{multline*} \left\langle \St_{L_{i+1}}(\be_1),\St_{L_{i+1}}(\be_2) \right\rangle=\left\langle\beta_{i+1}(G_{i}),\beta_{i+1}(G_{i}^{a_i}) \right\rangle =
\beta_{i+1}\left(\langle G_{i},G_{i}^{a_i}\rangle\right)=\beta_{i+1}(M)=M_{i+1}.
\end{multline*}
Since $G_i \neq M$, we have $\beta_{i+1}(G_i) \neq M_{i+1}$, and so $\be_1 \neq \be_2$ in $T_{i+1}$. Thus (2) holds.

To prove (3), observe that if $c \in G_i\setminus G_{i-1}$ then  $e_1$ is folded with $c\circ e_1$ at Step 2, hence
$\f_i(e_1)=\f_i(c\circ e_1)$ in $T_{i+1}$.

Now, suppose that $c \in M\setminus G_i$ (then the images of the edges $e_1$ and $c\circ e_1$ after the folds at Step 2
are distinct). Since  the restriction of $\phi_i:L_i \to L_{i+1}$ to $M$ is injective by Lemma  \ref{lem:fold_props}.(i),
$c \notin G_i$ implies that $\phi_i(c) \notin \phi_i(G_i)=\St_{L_{i+1}}(\be_1)$ (see property (2)).
Therefore $\f_i(c \circ e_1)=\phi_i(c) \circ \be_1 \neq \be_1$ in $T_{i+1}$, as required.

The proof of property (4) is similar to the the proof of (3), and is left as an exercise for the reader.
\end{proof}

\section{Showing that the convergence is strong} \label{sec:strong_conv}
In this section we will prove that if $M$ satisfies the condition \descref{(P4)}, described below, in addition to \descref{(P1)},\descref{(P2)} from
Subsection \ref{subsec:L_i}, then the sequence $(T_i,\f_i)_{i \in \N}$ converges strongly in the category of $\R$-trees.
We will then check that the limit group $L$, defined in Subsection~\ref{subsec:(FA)}, acts on the resulting limit $\R$-tree $T$ so that the stabilizers of arcs are
isomorphic to subgroups of $G_n$, $n \in \N\cup\{0\}$.

Further in this section we will assume that, in addition to properties \descref{(P1)} and \descref{(P2)} (\descref{(P3)} is not needed here),
the finitely generated group $M$, its ascending chain of subgroups $G_0<G_1<\dots$ and elements $a_i \in M$ also satisfy the following condition:
\begin{itemize}
\descitem{(P4)} \normalfont  for all $i \in \N$ and $c \in G_i \setminus G_{i-1}$, neither $\langle G_i, G_i^{ a_i c a_i^{-1}} \rangle$ nor $\langle G_i, G_i^{a_i^{-1} c a_i} \rangle$ is contained in a conjugate of
$G_n$ in $M$ for any $n \in \N \cup \{0\}$.
\end{itemize}

Consider the sequence $(L_i,T_i)_{i \in \N}$, of $\R$-trees with symmetry, together with the morphisms $(\phi_i,\f_i): (L_i,T_i) \to (L_{i+1},T_{i+1})$, $i \in \N$,
constructed in Section \ref{sec:folding_seq}. The construction together with surjectivity of $\phi_i:L_i \to L_{i+1}$ (see Lemma \ref{lem:fold_props}.(ii)) imply
 that each map $\f_i: T_i \to T_{i+1}$ is surjective.
For $1 \le i < j$, let $\f_{ij}:T_i \to T_j$ denote the $\R$-tree morphisms given by $\f_{ij}:=\f_{j-1} \circ\dots\circ \f_i$. These maps are equivariant with respect to the epimorphisms
$\phi_{ij}:L_i \to L_j$ which have already been defined in Subsection \ref{subsec:(FA)}. For convenience of notation, we let $\f_{ii}:T_i \to T_i$ be the identity map.

\begin{lemma}\label{lem:aux-1} Let $e_1$ and $e_2$ be two distinct edges of the tree $T_i$, for some $i \in \N$,
which are adjacent to the same vertex $v={e_1}_-={e_2}_-$ of $T_i$.
Suppose that the subgroup  of $\St_{L_i}(v)$ generated by $\St_{L_i}(e_1)$ and $\St_{L_i}(e_2)$ is not contained in a conjugate of
$G_n$ or in a conjugate of $\beta_i(G_n)$ in $L_i$, for any $n \in \N\cup \{0\}$. Then for every $j > i$, $\f_{ij}(e_1) \cap \f_{ij}(e_2)=\{\f_{ij}(v)\}$ in $T_j$.
\end{lemma}

\begin{proof} Since the action of $L_i$ on $T_i$ has exactly two orbits of vertices, we can assume that either $\St_{L_i}(v)=M$ or $\St_{L_i}(v)=M_i$.

Arguing by contradiction, suppose that for some $j>i$, $\f_{ij}(e_1) \cap \f_{ij}(e_2)$ is strictly larger than $\f_{ij}(v)$ in the simplicial tree
$T_j$. Then this intersection  must contain at least one edge $f$ of $T_j$, which is adjacent to $\f_{ij}(v)$ (as $\f_{ij}(e_1)$ and $\f_{ij}(e_2)$ are simplicial subtrees of $T_j$ by construction).
Then $\phi_{ij}\left(\St_{L_i}(e_1)\right)$ and $\phi_{ij}\left(\St_{L_i}(e_2)\right)$
will both stabilize $f$ in $T_j$, i.e.,
\begin{equation}\label{eq:edge_stabs_gen}
\left\langle \phi_{ij}\left(\St_{L_i}(e_1)\right), \phi_{ij}\left(\St_{L_i}(e_2)\right) \right\rangle
=\phi_{ij}\left(\left\langle\St_{L_i}(e_1), \St_{L_i}(e_2) \right\rangle\right)\leqslant \St_{L_j}(f) \leqslant
\St_{L_j}(\f_{ij}(v)).
\end{equation}
If  $\St_{L_i}(v)=M$ then $ \St_{L_j}(\f_{ij}(v))=M$ and $\St_{L_j}(f)=G_{j-1}^h$ for some $h \in M$.  Since
$\phi_{ij}$ induces the identity map between the stabilizer of $v$ in $T_i$ and the stabilizer of $\f_{ij}(v)$ in $T_j$ (see Lemma \ref{lem:fold_props}),
we can use \eqref{eq:edge_stabs_gen} to conclude that $\left\langle\St_{L_i}(e_1), \St_{L_i}(e_2) \right\rangle \subseteq G_{j-1}^h$ in $M$ (and hence in $L_i$), contradicting the assumptions.

So, suppose that
$\St_{L_i}(v)=M_i$ in $L_i$. Then, by Lemma \ref{lem:fold_props},  $ \St_{L_j}(\f_{ij}(v))=M^b$ for some $b \in L_j$, $\St_{L_j}(f)=(G_{j-1}^{b})^h$ for some $h \in M^b$, and
$\phi_{ij}$ induces an isomorphism between $M_i$ and $M^b$, which maps conjugates of $\beta_i(G_{j-1})$ in $M_i$ to conjugates of $G_{j-1}^b$ in $M^b$.
Hence \eqref{eq:edge_stabs_gen} shows that the subgroup $\left\langle\St_{L_i}(e_1), \St_{L_i}(e_2) \right\rangle$ is contained in a conjugate of
$\beta_i(G_{j-1})$ in $M_i$ (and thus in $L_i$), which, again, leads to a contradictions with the assumptions.
\end{proof}

Suppose that $S_1$ and $S_2$ are (simplicial) subtrees of the tree $T_i$ for some $i \in \N$. We will say that
\emph{a folding happens between $S_1$ and $S_2$ at stage $j$},
for some $j>i$, if 
the intersection $\f_{i,j}(S_1) \cap \f_{i,j}(S_2)$, of the images of $S_1$ and $S_2$ in $T_j$,
is strictly larger than the $\f_{j-1}$-image of the intersection of their images in $T_{j-1}$, i.e.,
$$\f_{j-1} \left(\f_{i,j-1}(S_1) \cap \f_{i,j-1}(S_2)\right) \subsetneqq \f_{i,j}(S_1) \cap \f_{i,j}(S_2) \mbox{ in } T_j.$$

Recall that each tree $T_i$ is equipped with the metric $\d_i$, which is obtained from the standard simplicial metric after downscaling by $2^{i-1}$. In other words, every edge of
$T_i$ is proclaimed to be isometric to the interval $[0,1/2^{i-1}]$. This takes into account the edge subdivision that occurs in our morphisms, making sure that
the $\d_i$-distance between two endpoints of an edge from $T_i$ is equal the $\d_{i+1}$-distance between the images of
these endpoints in $T_{i+1}$: see the lemma below.

\begin{lemma}\label{lem:aux-2} If $1 \le i \le j$ then the restriction of the map $\f_{ij}:T_i \to T_j$ to any edge $e$ of $T_i$ is injective, and thus it induces an isometric embedding of $e$
in $T_j$ with respect to the metrics $\d_i$ on $T_i$ and $\d_j$ on $T_j$.
\end{lemma}

\begin{proof} Since $T_i$ has only one orbit of edges, we can assume that $e$ is the edge from the fundamental region, and so $\St_{L_i}(e)=G_{i-1}$.
First, note that $\f_{i,i+1}=\f_i$ and, according to Lemma \ref{lem:morph_summary},
the image of $e$ in $T_{i+1}$ is subdivided into two distinct edges $\f_i(e)=\be_1 \cup \be_2$,
which are adjacent to the vertex $v$ that is the image of the midpoint of $e$ in $T_{i+1}$.
Therefore the restriction of the map $\f_i:T_i \to T_{i+1}$ to $e$ is injective.

Now we  show that $M_{i+1}$ cannot be contained in any conjugate of $G_n$ or $\beta_{i+1}(G_n)$ in $L_{i+1}$ for any $n \in \N \cup \{0\}$. Indeed, if
$M_{i+1} \subseteq G_n^h$ for some $h \in L_{i+1}$ then $M_{i+1}$ would fix both $v$ and $h \circ u$, where $u$ is the vertex of $T_{i+1}$ fixed by $M$ (as $G_n \leqslant M$). Moreover,
$v \neq h \circ u$, as $v$ and $u$ lie in different $L_{i+1}$-orbits, implying that
$M_{i+1}$ must fix an edge adjacent to $v$ in $T_{i+1}$. The latter is impossible as $M_{i+1}$ is strictly larger than $\beta_{i+1}(G_i)^g$ for any $g\in M_{i+1}$.
On the other hand, if $M_{i+1} \subseteq \beta_{i+1}(G_n)^h$ for some $h \in L_{i+1}$, then, clearly, $h \notin M_{i+1}=\St_{L_{i+1}}(v)$,
and so $M_{i+1}$ fixes two distinct vertices $v$ and  $h \circ v$ in $T_{i+1}$. The latter again contradicts the fact that $M_{i+1}$ does not fix any edge of $T_{i+1}$.

Therefore we can apply Lemma \ref{lem:aux-1} to conclude that for any $j \ge i+1$ one has
\begin{equation}\label{eq:fe_1-e_2}
\f_{i+1,j}(\be_1)\cap\f_{i+1,j}(\be_2)=\{\f_{i+1,j}(v)\} \mbox { in } T_j.
\end{equation}
Thus no folding can happen between $\be_1$ and $\be_2$ at any stage $j >i+1$.

Let $j>i$; we will show that the restriction of $\f_{ij}$ to $e$ is injective by induction on
$j-i$. The case when $j-i=1$ has already been considered, so assume that $j>i+1$. Suppose that $\f_{ij}(u)=\f_{ij}(w)$ for two distinct points $u,w \in e$,
such that (without loss of generality) $\f_i(u) \neq v$.
Since $j-(i+1)<j-i$, by the induction hypothesis the restriction of $\f_{i+1,j}$ to each of $\be_1$ and $\be_2$ is injective. Hence
$\f_{ij}(u)=\f_{i+1,j}(\f_i(u)) \neq \f_{i+1,j}(v)$; by the same reason $\f_i(u) \neq \f_i(w)$ cannot both belong to $\be_1$ or $\be_2$.
Therefore $\f_{i+1,j}(v) $ and $\f_{ij}(u)$ are two distinct points of the intersection $\f_{i+1,j}(\be_1)\cap\f_{i+1,j}(\be_2)$, which contradicts  \eqref{eq:fe_1-e_2}.
Thus $\f_{ij}$ induces an isometry of $e$ with its image in $T_j$.
\end{proof}

\begin{lemma}\label{lem:folds_btw_edges} Let $a$ and $b$ be two distinct edges of $T_l$ for some $l \in \N$. Then there can be no more than four different stages at which foldings
happen between  $a$ and $b$.
\end{lemma}

\begin{proof} Suppose that $k \in \N$, $k >l$, is a stage by which two different foldings between $a$ and $b$ have already occurred. Then, by Lemma \ref{lem:aux-2},
$\f_{lk}(a)$ and $\f_{lk}(b)$ are simple simplicial paths in $T_k$ and the intersection $\f_{lk}(a)\cap \f_{lk}(b)$ is a geodesic segment $[u,w]$ for some vertices $u$ and
$w$ of $T_k$, $u \neq w$.
Let $a_-,a_+$ and $b_-,b_+$ denote the endpoints of $a$ and $b$, respectively, so that
$[u,w]=[\f_{lk}(a_-),w] \cap [\f_{lk}(b_-),w]= [u,\f_{lk}(a_+)] \cap [u,\f_{lk}(b_+)]$ in $T_k$.

By Lemma \ref{lem:aux-2}, any folding happening between $a$ and $b$ at any stage $j>k$ has to come either from a folding between the geodesic segments
$[u,\f_{lk}(a_-)]$ and $[u,\f_{lk}(b_-)]$, or from a folding between $[w,\f_{lk}(a_+)]$ and $[w,\f_{lk}(b_+)]$.
If a folding happens between $p_1:=[u,\f_{lk}(a_-)]$ and $p_2:=[u,\f_{lk}(b_-)]$ at some stage $j>k$, then pick minimal such $j$. Then the restriction of the map
$\f_{ki}$ to the union $p_1 \cup p_2$ is still injective, where $i:=j-1$, and so $\f_{ki}(p_1)\cap \f_{ki}(p_2)=\{\f_{ki}(u)\}$ in $T_i$.

Note that $v:=\f_{ki}(u)$ is a vertex of $T_i$. Let $e$ denote the first edge of $\f_{ki}(p_1)$, and let $f$ denote the first edge of $\f_{ki}(p_2)$ in $T_i$;
thus $e_-=f_-=v$ (see Figure \ref{fig:fork}).

\begin{figure}[ht]\begin{center}
\input{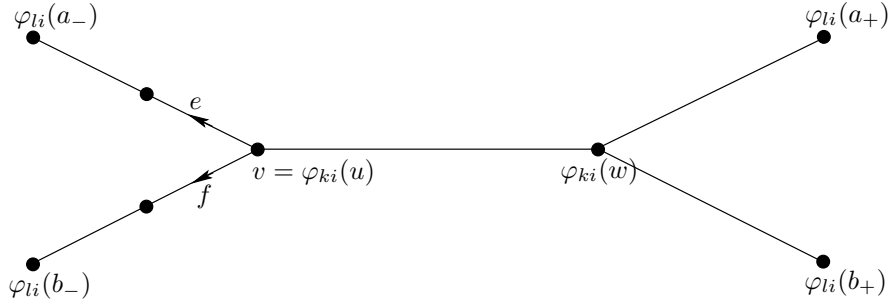}
\caption{\label{fig:fork} The images of the edges $a$ and $b$ in the tree $T_i$. }
\end{center}
\end{figure}

Since $T_i$ has only one orbit of edges we can assume that $\St_{L_i}(e)=G_{i-1}$ and either $\St_{L_i}(v)=M$ or $\St_{L_i}(v)=M_{i}$. Let us suppose that
$\St_{L_i}(v)=M$ (the other case is similar). Then
$f=c \circ e$ for some $c \in \St_{L_i}(v)=M$. Note that $c \notin G_{i-1}=\St_{L_i}(e)$ as $e \neq f$.
Now let us recall how the morphism from $T_i$ to $T_{i+1}=T_j$ works.
First we subdivide the edge $e$ into two halves $e_1$ and $e_2$, such that ${e_1}_-=e_-=v$, ${e_1}_+={e_2}_-$ and ${e_2}_+=e_+$.
Then $f$ is subdivided into the union of $f_1=c \circ e_1$ and $f_2=c \circ e_2$.

If $c \notin G_i$ then, according to Lemma \ref{lem:morph_summary}.(3), the images of $e_1$ and $f_1$ in $T_{i+1}$ are distinct, which
means that no folding between $p_1$ and $p_2$ can happen at stage $j=i+1$, contradicting the choice of $j$.

Hence $c \in G_i\setminus G_{i-1}$, in which case $\f_i(e_1)=\f_i(f_1)$ in $L_{i+1}$ by Lemma \ref{lem:morph_summary}.(3).
Since $\St_{L_{i+1}}(\f_i(e_2))=\beta_{i+1}(G_i^{a_i}) \leqslant M_{i+1}$ (by Lemma \ref{lem:morph_summary}.(2)) and $\phi_i(c)=c \in M$,
we have $$\St_{L_{i+1}}(\f_i(f_2))=\St_{L_{i+1}}(\phi_i(c) \circ \f_i(e_2))=\beta_{i+1}(G_i^{a_i})^c=\beta_{i+1}(G_i^{ca_i}),$$
where the last equality holds because $c \in G_i$ is identified with $\beta_{i+1}(c) \in \beta_{i+1}(G_i)$ in $L_{i+1}$ by the definition of $L_{i+1}$.
It follows that
\begin{equation}\label{eq:e_2-f_2}
\left\langle\St_{L_{i+1}}(\f_i(e_2)),\St_{L_{i+1}}(\f_i(f_2))\right\rangle = \beta_{i+1} \left(\langle G_i^{a_i},G_i^{ca_i} \rangle\right) =\beta_{i+1}\bigl(\langle G_i,G_i^{a_i^{-1}ca_i} \rangle\bigr)^{\beta_{i+1}(a_i)}.
\end{equation}
Recalling that $c \in G_i\setminus G_{i-1}$, we see that $\beta_{i+1} \bigl(\langle G_i,G_i^{a_i^{-1}ca_i} \rangle\bigr)$ is not contained in a conjugate of $\beta_{i+1}(G_n)$ in
$M_{i+1}=\St_{L_{i+1}}({e_2}_-)$, for any $n \in \N\cup \{0\}$, by \descref{(P4)}.
Since the stabilizer of any edge adjacent to ${e_2}_-$ in $T_{i+1}$ is conjugate to $\beta_{i+1}(G_i)$ in $M_{i+1}$,
we can argue in the same way as in the proof
of Lemma \ref{lem:aux-2}, to show that $\beta_{i+1} \bigl(\langle G_i,G_i^{a_i^{-1}ca_i} \rangle\bigr)$ cannot be contained in a conjugate of $G_n$ or in a conjugate of
$\beta_{i+1}(G_n)$ in $L_{i+1}$, for any $n \in \N\cup \{0\}$. In view of \eqref{eq:e_2-f_2}, the latter allows us to apply Lemma \ref{lem:aux-1},
concluding that no further folding can happen between $p_1$ and $p_2$ at any stage $m$ with $m>i+1=j$.

Thus at most one folding is possible between $p_1:=[u,\f_{lk}(a_-)]$ and $p_2:=[u,\f_{lk}(b_-)]$. Similarly, at most one folding is possible between
$[w,\f_{lk}(a_+)]$ and $[w,\f_{lk}(b_+)]$. This shows that there can be no more than four different stages when foldings happen between $a$ and $b$, as claimed.
\end{proof}

\begin{lemma}\label{lem:edge_stab} Let $e$ be an edge of $T_i$, for some $i \in \N$. Then for any $j \ge i$, $\St_{L_j}(\f_{ij}(e))=\phi_{ij}(\St_{L_i}(e))$.
\end{lemma}

\begin{proof} The statement will again be proved by induction on $j-i$. 
Assume, first, that $j=i+1$. Since
$T_i$ contains only one orbit of edges under the natural action of $L_i$, we can assume that $\St_{L_i}(e)=G_{i-1}$ and $\St_{L_i}(e_-)=M$, for some endpoint $e_-$ of $e$.
By the construction, $\f_{i}(e)=e_1\cup e_2$ in $T_{i+1}$ with $\St_{L_{i+1}}(e_1)=\beta_{i+1}(G_i)$ and $\St_{L_{i+1}}(e_2)=\beta_{i+1}(G_i^{a_i})$
in $M_{i+1}=\St_{L_{i+1}}(v)$, where $v$ the vertex of $T_{i+1}$, adjacent to both $e_1$ and $e_2$, which is the image of the midpoint of $e$ in $T_{i+1}$.
Since $\beta_{i+1}: M \to M_{i+1}$ is injective, we can observe that
\begin{equation}\label{eq:edge-stab}
\St_{L_{i+1}}(\f_{i}(e))=\St_{L_{i+1}}(e_1) \cap\St_{L_{i+1}}(e_2)=\beta_{i+1}(G_i)\cap \beta_{i+1}(G_i^{a_i})=\beta_{i+1}(G_i \cap G_i^{a_i}).
\end{equation}

Now, let us show that $G_i \cap G_i^{a_i}=G_{i-1}$ in $M$. Indeed, $G_{i-1} \subseteq G_i \cap G_i^{a_i}$ by \descref{(P1)}, and if there existed $d \in \left( G_i \cap G_i^{a_i}\right)\setminus G_{i-1}$, then
$d= a_i c a_i^{-1}$ for some $c \in G_i\setminus G_{i-1}$. Hence $G_i=G_i^d=G_i^{a_ica_i^{-1}}$, and so $\langle G_i,G_i^{a_ica_i^{-1}}\rangle=G_i$, contradicting \descref{(P4)}. Therefore
$G_i \cap G_i^{a_i}=G_{i-1}$, and \eqref{eq:edge-stab} gives $\St_{L_{i+1}}(\f_{i}(e))=\beta_{i+1}(G_{i-1})$.
But $\beta_{i+1}(G_{i-1})=G_{i-1}=\phi_i(G_{i-1})=\phi_i(\St_{L_i}(e))$ by the definitions of $L_{i+1}$ and $\phi_i$. Hence
\begin{equation}\label{eq:st-1}
\St_{L_{i+1}}(\f_{i}(e))=\St_{L_{i+1}}(e_1) \cap\St_{L_{i+1}}(e_2)=\phi_i(\St_{L_i}(e)) \mbox{ in } L_{i+1}.
\end{equation}

Thus we can now assume that $j>i+1$. Then
\begin{equation}\label{eq:st-2}
\St_{L_j}(\f_{ij}(e))=\St_{L_j}(\f_{i+1,j}(e_1)) \cap \St_{L_j}(\f_{i+1,j}(e_2)) \mbox { in } L_j.
\end{equation}
Since $j-(i+1)<j-i$,  the induction hypothesis implies that
\begin{equation}\label{eq:st-3}
\St_{L_j}(\f_{i+1,j}(e_1))=\phi_{i+1,j}(\St_{L_{i+1}}(e_1))    \mbox{ and } \St_{L_j}(\f_{i+1,j}(e_2))=\phi_{i+1,j}(\St_{L_{i+1}}(e_2)).
\end{equation}

It remains to note that $\St_{L_{i+1}}(e_1)$ and $\St_{L_{i+1}}(e_2)$ are both subgroups of
$M_{i+1}=\St_{L_{i+1}}(v)$ and $\phi_{i+1,j}$ is injective on $M_{i+1}$ by Lemma \ref{lem:fold_props}.(i), hence
\begin{equation}\label{eq:st-4}
\phi_{i+1,j}(\St_{L_{i+1}}(e_1)) \cap \phi_{i+1,j}(\St_{L_{i+1}}(e_2))=\phi_{i+1,j}(\St_{L_{i+1}}(e_1) \cap \St_{L_{i+1}}(e_2)).
\end{equation}
Collecting the equalities \eqref{eq:st-2}-\eqref{eq:st-4} together and recalling \eqref{eq:st-1}, we achieve
$$\St_{L_j}(\f_{ij}(e))=\phi_{i+1,j}(\St_{L_{i+1}}(e_1) \cap \St_{L_{i+1}}(e_2))=\phi_{i+1,j}(\phi_i(\St_{L_i}(e)))=\phi_{ij}(\St_{L_i}(e)),$$
as required.
\end{proof}

\begin{prop} \label{prop:strong-lim} The sequence of simplicial $\R$-trees $(T_i,\d_i,\f_i)_{i \in \N}$ defined above is strongly convergent.
\end{prop}

\begin{proof} Consider any $l \in \N$ and any points $x,y$ in $T_l$. 
Let $p$ be some finite simplicial path in $T_l$ containing $x$ and $y$. By Lemma \ref{lem:aux-2}, for every $i \in \N$, $i \ge l$, the restriction of $\f_{li}$ to each edge of $p$ is injective, and by
Lemma \ref{lem:folds_btw_edges}, for any pair of edges $a$ and $b$ of $p$, there exists $K=K(a,b) \in \N$ such that restriction of $\f_{ij}$ to
$\f_{li}(a) \cup \f_{li}(b)$ is injective, provided $j\ge i \ge K$. Since $p$ contains only finitely many edges (in $T_l$), we can
conclude that the restriction of $\f_{kj}$ to $\f_{lk}(p)$ is injective for any $j \ge k$, where $k:=\max\{K(a,b) \mid a,b \mbox{ are edges of }p\}$. Therefore
$\d_j(\f_{lj}(s),\f_{lj}(t))=\d_k(\f_{lk}(s),\f_{lk}(t))$ for any points $s,t \in p$ and any $j \ge k$.
\end{proof}

Since the sequence $(T_i,\d_i,\f_i)$ converges strongly, we can form the limit $\R$-tree $(T,\d)$, as discussed in Subsection \ref{subsec:strong_lim}.
Keeping the same notation, we let $\theta_i:(T_i,\d_i) \to (T,\d)$, $i \in \N$,  denote the resulting $\R$-tree morphisms.
We will also use the pseudometric $\hat \d$ and the equivalence relation $\sim$ on $T_1$ defined in Subsection~\ref{subsec:strong_lim}.

It is easy to see that the group $L_1$ acts by isometries on the $\R$-tree $(T,\d)$ in the following manner.
If $g \in L_1$ and $\bar x\in T$, then pick any $x \in T_1$ with $\theta_1(x)=\bar x$ and define $g \circ \bar x:=\theta_1(g \circ x) \in T$.

Let $L$ be the direct limit of the sequence $(L_i,\phi_i)_{i \in \N}$ (see Subsection \ref{subsec:(FA)}). We are finally ready to prove the main result of this section:

\begin{thm} \label{thm:small_arc_stab}
The group $L$ acts on the $\R$-tree $(T,\d)$ non-trivially and by isometries.
Moreover, given two distinct points $\bar x,\bar y$ of $T$, there exists $m \in \N $ such that the pointwise $L$-stabilizer of the geodesic segment $[\bar x,\bar y]$ is isomorphic to a subgroup of
$G_{m-1}$.
\end{thm}

\begin{proof} By definition, $L=L_1/N$ for the normal subgroup $N=\bigcup_{i=2}^\infty \ker(\phi_{1i})\lhd L_1$. The natural action of $L_i$ on $T_i$ induces an action
of $L_1$ on $T_i$, for which every element $h \in \ker(\phi_{1i})$ acts as identity on $T_i$. Consider any point $x \in T_1$ and any $h \in N$. Then $h \in \ker(\phi_{1i})$ for some
$i \ge 2$, hence $\d_j(\f_{1j}(x),\f_{1j}(h \circ x))=0$ for all $j \ge i$. Therefore $h\circ x \sim x$, thus $h$ acts as identity on $T$.
Therefore the above action of $L_1$ on $T$ naturally induces an isometric action of $L=L_1/N$ on $T$.
If this action was trivial, then there would exist a point $y \in T_1$ such that $\hat\d(y,g \circ y)=0$ for all $g \in L_1$. Let $\{g_1,\dots,g_n\}$ be some finite
generating set of $L_1$. By Proposition \ref{prop:strong-lim}, there exists $j \in \N$ such that $\d_j(\f_{1j}(y),\f_{1j}(g_l\circ y))=
\d_j(\f_{1j}(y),\phi_{1j}(g_l)\circ\f_{1j}(y))=0$ for any $l \in \{1,\dots, n\}$.
Hence the point $\f_{1j}(y) \in T_j$ is fixed by $\phi_{1j}(g_1), \dots,\phi_{1j}(g_n)$, which generate $L_j=\phi_{1j}(L_1)$.
This contradicts the fact that the action of $L_j$ on $T_j$ is
non-trivial. Therefore the action of $L$ on $T$ must also be non-trivial.

Suppose that $\bar x,\bar y$ are two distinct points in $T$ and $x,y$ are some preimages of $\bar x, \bar y$ in $T_1$ respectively.
By Proposition \ref{prop:strong-lim}, there is $k \in \N$ such that for any $m \ge k$ the restriction of the natural map $\theta_m:(T_m,\d_m) \to (T,\d)$
to $[\f_{1m}(x), \f_{1m}(y)]$ is an isometry.  Choose some $m \ge k$ so that $2^{2-m}<\d(\bar x,\bar y)$. Then the distance $\d_m(\f_{1m}(x), \f_{1m}(y))=\d(\bar x,\bar y)$
is greater than twice the edge length in $T_m$, therefore the geodesic segment $[\f_{1m}(x), \f_{1m}(y)]$ contains some edge $e$ of $T_m$. It follows that
$\bar e:=\theta_m(e)$ is contained in the geodesic segment $[\bar x, \bar y]=\theta_m([\f_{1m}(x), \f_{1m}(y)])$ of $T$.
Evidently, $\St_L([\bar x,\bar y])\leqslant \St_L(\bar e)$, so it remains to show that $\St_L(\bar e)$ is isomorphic to a subgroup of $G_{m-1}$.

Assume that $\bar g \in \St_L(\bar e)$ and take any $g \in L_1$ with $\psi_1(g)=\bar g$, where the epimorphisms $\psi_i:L_i \to L$, $i \in \N$,
were defined in Subsection \ref{subsec:(FA)}. Choose any points $s,t\in T_1$ that
are preimages of the endpoints $e_-$ and $e_+$ of $e$ respectively.
Since $\bar g \in \St_L(\bar e)$, by the definition of the action of $L$ on $T$, the element $g \in L_1$ must fix ${\bar e}_-=\theta_m(e_-)$ and ${\bar e}_+=\theta_m(e_+)$ in $T$, thus
$g\circ s \sim s$ and $g \circ t \sim t$ in $T_1$. Therefore, according to Proposition~\ref{prop:strong-lim}, there exists
$j \ge m$ such that $\d_j(\f_{1j}(g \circ s),\f_{1j}(s))=0$ and $\d_j(\f_{1j}(g \circ t),\f_{1j}(t))=0$ in $T_j$. Consequently
$\phi_{1j}(g) \circ \f_{1j}(s)=\f_{1j}(g \circ s)=\f_{1j}(s)$ and
$\phi_{1j}(g) \circ \f_{1j}(t)=\f_{1j}(g \circ t)=\f_{1j}(t)$ in $T_j$, yielding that
$\phi_{1j}(g) \circ \f_{mj}(e_-)=\f_{mj}(e_-)$ and $\phi_{1j}(g) \circ \f_{mj}(e_+)=\f_{mj}(e_+)$ in $T_j$. Recall that $\f_{mj}(e)$ is a simple path in the tree $T_j$
by Lemma~\ref{lem:aux-2}, so it is completely determined by its endpoints, and thus $\phi_{1j}(g) \circ \f_{mj}(e)=\f_{mj}(e)$.
Now we can apply Lemma \ref{lem:edge_stab}, claiming that
there exists $h \in \St_{L_m}(e)$ such that $\phi_{1j}(g)=\phi_{mj}(h)$. Therefore, in view of \eqref{eq:compos-phi}, we get
$$\bar g=\psi_1(g)=\psi_j(\phi_{1j}(g))=\psi_j(\phi_{mj}(h))=\psi_m(h),$$
which shows that $\St_L(\bar e) \leqslant \psi_m\left(\St_{L_m}(e) \right)$. Since $\psi_m$ is injective on vertex and edge stabilizers  for the action of $L_m$ on $T_m$
(this follows from Lemma \ref{lem:fold_props}.(i)),
we can conclude that $\psi_m\left(\St_{L_m}(e) \right) \cong \St_{L_m}(e) \cong G_{m-1}$, as claimed.
\end{proof}

\section{Construction of a suitable group $M$} \label{sec:M}
In this section we suggest a construction of a finitely generated group $M$ together with its ascending sequence of subgroups
$G_0<G_1<\dots$ and elements $a_i \in M$, $i \in \N$ that satisfy properties \descref{(P1)}-\descref{(P4)} above. (Unfortunately Thompson's group $V$, together with
its subgroups $G_i$ and elements $a_i$, discussed in Subsection \ref{subsec:V}, does not enjoy \descref{(P4)}. Indeed, given $i \in \N$, choose  any
$c \in \St_V([0,3/2^{i+2}))\subset G_i$ such that $c \notin \St_V([0,1/2^i)) =G_{i-1}$. Then $a_i c a_i^{-1} \in \St_V([0,1/2^{i+2}))=G_{i+1}$, hence
$\langle G_i,G_i^{a_ica_i^{-1}} \rangle \subseteq G_{i+1}$ in $V$.)

The construction will be based on the small cancellation theory over (word) hyperbolic groups proposed by Gromov in \cite{Gromov} and developed by Olshanskii in \cite{Olsh-G_sbgps}.
For convenience we will actually utilize a generalization of Olshanskii's techniques obtained by the author in \cite{Min-pap3}.

Recall, that a group is said to be \emph{elementary} if it contains a cyclic subgroup of finite index; in particular any finite group is elementary.
For any non-elementary subgroup $H$ of a hyperbolic group $F$ there exists a unique maximal finite subgroup $E_F(H) \leqslant F$ that is normalized
by $H$ in $F$ (see \cite[Prop. 1]{Olsh-G_sbgps}).
Given a non-elementary hyperbolic group $F$, we will say that a subgroup $H\leqslant F$ is a \emph{$G$-subgroup}  if $H$ is non-elementary and
$E_F(H)=\{1\}$ (according to \cite[Thm. 1]{Olsh-G_sbgps}, this is a special case of Olshanskii's definition of a $G$-subgroup from \cite[p. 366]{Olsh-G_sbgps}).
Evidently, $E_F(F) \leqslant E_F(H)$ for any non-elementary subgroup $H$ of $F$. In particular, if $F$ contains at least one $G$-subgroup then $E_F(F)=\{1\}$.

Let $H$ be a subgroup of a group $F$ and let $Q \subseteq K$.  Following  \cite{Min-pap3} we will say that \emph{$Q$ is small
relative to $H$} if for any two finite subsets $P_1, P_2 \subseteq F$, $H$ is not contained in the product $P_1 Q^{-1} Q P_2$ in $F$.

Given a hyperbolic group $F$ with a fixed finite generating set $X$, let $\Gamma(F,X)$ denote the Cayley graph of $F$ with respect to $X$.
Recall also that a subset $Q$ of $F$
(or of $\Gamma(F,X)$) is said to be \emph{quasiconvex}  if there exists $\varepsilon >0$ such that for any pair of
elements $u, v \in Q$ and any geodesic segment $p$ connecting $u$ and $v$, $p$ belongs to a
closed $\varepsilon$-neighborhood of $Q$ in $\Gamma(F,X)$. It is well known that quasiconvexity of a subset is independent of the choice of the finite generating set $X$ of $F$
(see \cite{Gromov}).

The following statement is a special case of \cite[Thm. 1]{Min-pap3}:
\begin{lemma} \label{lem:g-sbgps} Let $H_1$, $H_2$ be $G$-subgroups of a non-elementary hyperbolic group
$F$. Assume that $Q \subseteq F$ is a quasiconvex subset which is small relative to $H_i$, $i=1,2$.
Then there exist a group $K$ and an epimorphism $\xi: F \to K$ such that
\begin{itemize}
  \item[(i)] $K$ is a non-elementary hyperbolic group;
  \item[(ii)] $\xi$ is injective on $Q$;
  \item[(iii)] $\xi(H_1)=\xi(H_2)=K$;
  \item[(iv)] $E_K(K)=\{1\}$.
\end{itemize}
\end{lemma}

Below we will only be interested in the case when the quasiconvex subset $Q$ is a union of finitely many quasiconvex subgroups.
In this case smallness of $Q$ relative to $H$ is easy to check (see \cite[Thm. 3]{Min-pap3}):

\begin{lemma}\label{lem:small} Suppose that $C_1$, $\dots$, $C_k$ are quasiconvex subgroups of a hyperbolic group $F$ and $H \leqslant F$ is an arbitrary subgroup.
Let $Q:=\bigcup_{i=1}^k C_i \subseteq F$. Then $Q$ is small relative to $H$ provided $|H:(H \cap C_i^f)|=\infty$ for every $i=1,\dots,k$ and each $f \in F$.
\end{lemma}

It is obvious that any finite subgroup of a hyperbolic group is quasiconvex, and it is well known that any infinite cyclic subgroup is quasiconvex
(see, for example, \cite[Cor. 3.4]{Mih}). Since the union of a finite collection of quasiconvex subsets is again quasiconvex
(see \cite[Prop. 3.14]{ZG} or \cite[Lemma~2.1]{Min-pap1}), we can conclude that in any hyperbolic group $F$ a finite union of elementary subgroups is quasiconvex.

The required group $M$ will be obtained as a direct limit of a sequence of hyperbolic groups $K_j$, $j=0,1,2,\dots$.
We start with a strictly increasing sequence $G_0<G_1< G_2 < \dots$ of \emph{finite} groups such that $|G_1|>2$ and the following condition is satisfied:
\begin{equation}\label{eq:finite_gp_prop}
\mbox{ for each } i \in \N, \mbox{ if } N \lhd G_i \mbox{ and } N \subseteq G_{i-1} \mbox{ then } N=\{1\},
\end{equation}
i.e., $G_{i-1}$ does not contain non-trivial normal subgroups of $G_i$. As a matter of convenience we will assume that $G_0=\{1\}$ is the trivial group, and
we will let $\gamma_{i-1}:G_{i-1} \to G_{i}$ denote the embedding of $G_{i-1}$ into $G_i$, $i \in \N$.

The obvious choice would be to take $G_i$'s as a sequence of finite simple groups: e.g., $G_i=\alt(i+4)$ for $i=1,2,\dots$, equipped with the standard
embedding of $\alt(j)$ into $\alt(j+1)$ (as the subgroup leaving the last element of $\{1,2,\dots,j+1\}$ fixed).
On the opposite spectrum, one can choose $G_i$'s to be nilpotent, by letting $G_i=\mathrm{UT}(i+2,\mathbb{F})$  be the group of unitriangular matrices over a finite field $\mathbb{F}$, $i=1,2,\dots$, where
$\mathrm{UT}(j,\mathbb{F})$ is naturally embedded into $\mathrm{UT}(j+1,\mathbb{F})$ as the stabilizer of the last vector from the standard basis of $\mathbb{F}^{j+1}$.

Now, take  any non-elementary hyperbolic group $K_0$ with property (FA) (e.g., a hyperbolic triangle group -- see \cite[I.6.3, Ex. 5]{Serre}).
Without loss of generality we can suppose that
$E_{K_0}(K_0)=\{1\}$ (to achieve this, one can always replace $K_0$ with the quotient $K_0/E_{K_0}(K_0)$). We can also assume that $G_0=\{1\} \leqslant K_0$, and
take $Q_0:=\{1\} \subseteq K_0$.

\begin{lemma}\label{lem:K_j} There exist a sequence of groups $K_j$, $j \in \N$, epimorphisms
$\zeta_{j-1}:K_{j-1} \to K_{j}$, subsets $Q_j \subseteq K_j$ and elements $t_j \in K_j$  such that the following properties are satisfied for all $j \in \N$:
\begin{enumerate}[(a)]
  \item\label{it:a} $K_j$ is a non-elementary hyperbolic group with $E_{K_j}(K_j)=\{1\}$;
  \item\label{it:b} $Q_j$ is a finite union of elementary subgroups of $K_j$;
  \item\label{it:c} $\zeta_{j-1}$ is injective on $Q_{j-1}$, and $\zeta_{j-1}(Q_{j-1}) \subseteq Q_{j}$;
  \item\label{it:d} $G_j \leqslant K_j$ and $G_j \subseteq Q_j$;
  \item\label{it:e} $\zeta_{j-1}(G_{j-1})=G_{j-1}\subseteq G_{j}$ and $\zeta_{j-1}(g)=\gamma_{j-1}(g)$ for every $g \in G_{{j-1}}$;
  \item\label{it:f} $t_j$ centralizes $\zeta_{j-1}(G_{j-1})=G_{j-1}$ in $K_j$;
  \item\label{it:g} $K_j=\langle G_j,G_j^{t_j} \rangle$;
  \item\label{it:h}  for every $c \in G_j \setminus G_{j-1}$,  $c^{-1} c^{t_jct_j^{-1}}$ has infinite order in $K_j$, and
        $\langle c^{-1} c^{t_jct_j^{-1}} \rangle \subseteq Q_j$.
\end{enumerate}
\end{lemma}

\begin{proof} The group $K_0$ and the subset $Q_0\subseteq K_0$ have already been defined. So, arguing by induction we can assume that
for some $n \in \N$ we have already constructed the groups $K_0,\dots,K_{n-1}$, together with epimorphisms
$\zeta_{j-1}:K_{j-1} \to K_{j}$, subsets $Q_j \subseteq K_j$ and elements $t_j \in K_j$, $j=1,\dots,n-1$, satisfying properties \eqref{it:a}-\eqref{it:h} above.

In order to construct the group $K_n$, define an auxiliary group $F_n$ by the following presentation:
$$F_n= \langle K_{n-1},G_n,t_n \mid g=\gamma_{n-1}(g), ~t_n g t_n^{-1}=g \mbox{ for all }g \in G_{n-1}  \rangle.$$
In other words, $F_n$ is an HNN-extension of the free amalgamated product of $K_{n-1}$ with $G_n$ along $G_{n-1}=\gamma_{n-1}(G_{n-1})$ -- see Figure \ref{fig:F_n}.

\begin{figure}[ht]\begin{center}
\input{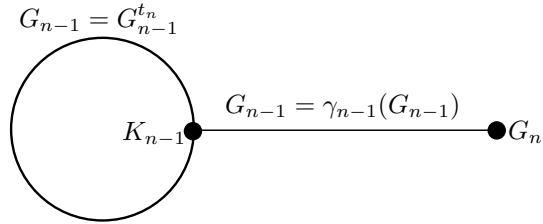}
\caption{\label{fig:F_n} The graph of groups for $F_n$.}
\end{center}
\end{figure}

According to this definition, $F_n$ is the fundamental group of a finite graph of groups with hyperbolic vertex groups
(as $K_{n-1}$ is hyperbolic by \eqref{it:a} and $|G_n|<\infty$) and finite edge groups (as $|G_{n-1}|<\infty$). Therefore $F_n$ is also a hyperbolic group
(e.g., by the Combination theorem of Bestvina and Feighn \cite{BF-comb, BF-correction}). Clearly $F_n$ is non-elementary and  $t_n$ centralizes $G_{n-1}$ in $F_n$.

Let $Q_n$ be the subset of $F_n$ defined by $$Q_n:=Q_{n-1}\cup G_n  \cup \bigcup_{c \in G_n\setminus G_{n-1}}
\langle c^{-1} c^{t_nct_n^{-1}} \rangle.$$
Then $Q_n$ is a finite union of elementary subgroups in $F_n$. Hence $Q_n$ is quasiconvex in $F_n$, and, by Lemma \ref{lem:small}, $Q_n$ is
small relative to any non-elementary subgroup $H \leqslant F_n$.

Let us check that  $K_{n-1}$ and $H:=\langle G_n,G_n^{t_n} \rangle$ are  $G$-subgroups of $F_n$. The subgroup $K_{n-1}$ is non-elementary by \eqref{it:a}.
On the other hand, it is easy to see that $H$ is isomorphic to the free amalgamated product $G_n*_{G_{n-1}=G_{n-1}^{t_n}}G_n^{t_n}$, which
contains non-abelian free subgroups because $|G_n:G_{n-1}|=|G_n^{t_n}:G_{n-1}^{t_n}|>2$ (by \eqref{eq:finite_gp_prop} as $|G_1|>2$) -- see \cite[Thm. 6.1]{Bass}. Therefore $H$ is also non-elementary.

In order to check the second part of the definition of a $G$-subgroup we will need the following auxiliary lemma:

\begin{lemma}\label{lem:f_n_sb} Suppose that $F$ is a group acting on a simplicial tree $S$ without edge inversions and $v$ is a vertex of $S$ such that
$H:=\St_F(v)$ is finitely generated and does not fix any edge of $S$. If $E \leqslant F$ is a finite subgroup normalized by $H$ in $F$ then $E \subseteq H$.
\end{lemma}
\begin{proof} Since $|E|<\infty$, the fixed point set $\Fix(E)$ is a non-empty convex subtree of $S$ (cf. \cite[I.6.3, Ex. 1]{Serre}) that is invariant under the action of $H$, as
$E$ is normalized by $H$. By the assumptions, $v\in \Fix(H)\neq\emptyset$, therefore every element $h \in H$ fixes some point of the subtree $\Fix(E)$
(cf. \cite[I.6.4, Cor. 2]{Serre}). Thus $H$ acts on the tree $\Fix(E)$ so that each element acts as an elliptic isometry. Since $H$ is finitely generated,
we can conclude that $H$ fixes some vertex $u \in \Fix(E)$ (see \cite[I.6.5, Cor. 3]{Serre}). But $v$ is the only vertex of $S$ fixed by $H$ because $H$ does
not fix any edge of $S$. Hence $v=u \in \Fix(E)$, implying that $E \subseteq \St_F(v)=H$, as claimed.
\end{proof}

Now let us continue the proof of Lemma \ref{lem:K_j}. The group $F_n$, constructed above, acts on the Bass-Serre tree $S$ corresponding to its natural
representation as a fundamental group of a graph of groups, and all edge stabilizers for this action are finite (as they are conjugates of $G_{n-1}$).
On the other hand, $K_{n-1}$ is infinite (by condition \eqref{it:a}) and so it cannot fix any edge of $S$, although it is the stabilizer of some vertex of $S$.
Hence we can apply  Lemma \ref{lem:f_n_sb} to conclude that $E_{F_n}(K_{n-1})\subseteq K_{n-1}$ in $F_n$.
It follows that $E_{F_n}(K_{n-1}) \subseteq E_{K_{n-1}}(K_{n-1})=\{1\}$ by \eqref{it:a};
thus $K_{n-1}$ is a $G$-subgroup of $F_n$.

Since $H=\langle G_n,G_n^{t_n} \rangle$ normalizes $E_{F_n}(H)$ and  $G_n,G_n^{t_n} \leqslant H$, we deduce that both $G_n$ and $G_n^{t_n}$ normalize
$E_{F_n}(H)$ in $F_n$. Recall that $G_n=\St_{F_n}(v)$ for some vertex $v$ of $S$, by definition, and so $G_n^{t_n}=\St_{F_n}(t_n \circ v)$. On the other hand,
neither $G_n$ nor $G_n^{t_n}$ fixes any edge of $S$ (as $|G_n|=|G_n^{t_n}|>|G_{n-1}|$), therefore $E_{F_n}(H)\subseteq G_n \cap G_n^{t_n}$ by Lemma  \ref{lem:f_n_sb}.
However, according to Britton's lemma for HNN-extensions (see \cite[Sec. IV.2]{L-S}),
$G_n \cap G_n^{t_n}=G_{n-1}$, so $E_{F_n}(H)$ is a normal subgroup of $G_n$ contained in $G_{n-1}$.
Hence, recalling \eqref{eq:finite_gp_prop}, we can conclude that $E_{F_n}(H)=\{1\}$, i.e., $H$ is a $G$-subgroup of $F_n$.

Thus all the assumptions of Lemma \ref{lem:g-sbgps} are verified, hence there exists a non-elementary hyperbolic group $K_n$ and an epimorphism $\xi_{n-1}:F_n \to K_n$
such that $\xi_{n-1}$ is injective on $Q_n$, $\xi_{n-1}(K_{n-1})=\xi_{n-1}(H)=K_n$ and $E_{K_n}(K_n)=\{1\}$. Let $\zeta_{n-1}:K_{n-1} \to K_n$ denote the restriction of
$\xi_{n-1}$ to $K_{n-1}$.
To simplify the notation we will identify $Q_n$, $G_n$ and $t_n$ with their $\xi_{n-1}$-images in $K_n$. It is now easy to check that the properties
\eqref{it:a}-\eqref{it:h} all hold for $j=n$. Indeed, the properties \eqref{it:a}-\eqref{it:f} are evident from construction and \eqref{it:g} follows because
$$K_n=\xi_{n-1}(H)=\xi_{n-1} \left(\langle G_n,G_n^{t_n} \rangle\right)=\langle G_n,G_n^{t_n} \rangle.$$

To establish \eqref{it:h} for $j=n$, we first observe that for every $c \in G_n\setminus G_{n-1}$ the element
$c^{-1} c^{t_nct_n^{-1}}$ 
has infinite order in $F_n$ (e.g., by applying Britton's lemma again). Now, since $\langle c^{-1} c^{t_nct_n^{-1}}\rangle \subseteq Q_n$ in $F_n$ and $\xi_{n-1}$
is injective on $Q_n$, we are able to conclude that the element $c^{-1} c^{t_nct_n^{-1}}=\xi_{n-1}(c^{-1} c^{t_nct_n^{-1}})$ still has infinite order and
$\langle c^{-1} c^{t_nct_n^{-1}}\rangle \subseteq \xi_{n-1}(Q_n)=Q_n$ in $K_n$.

Thus for every $j \in \N$ we have constructed the groups $K_j$ together with epimorphisms $\zeta_{j-1}:K_{j-1} \to K_{j}$,
subsets $Q_j \subseteq K_j$ and elements $t_j \in K_j$ that satisfy conditions \eqref{it:a}-\eqref{it:h} above.
\end{proof}

\begin{thm}\label{thm:M} There exists a finitely generated group $M$ which contains a strictly ascending sequence of subgroups
$G_0<G_1<\dots$ and elements $a_i \in M$, $i \in \N$, that satisfy the four properties \descref{(P1)}-\descref{(P4)} above.
\end{thm}

\begin{proof}
Define $M:=\lim_{j \to \infty}(K_j,\zeta_j)$ as the direct limit of the sequence $(K_j,\zeta_j)$ constructed in Lemma \ref{lem:K_j}.
Let $\tau_j: K_j \to M$ denote the canonical epimorphism, $j \in \N \cup\{0\}$. The properties \eqref{it:c} and \eqref{it:d} of Lemma \ref{lem:K_j}
imply that
$\tau_j$ is injective on $G_j$ and $Q_j$, therefore we will identify $G_j$ and its elements with their images in $M$ for every $j \in \N\cup \{0\}$.
Property \eqref{it:e} yields that $G_{j-1}< G_j$ in $M$, for all $j \in \N$. For every $j \in \N$ we let
$a_j:=\tau_j(t_j) \in M$. Then property \eqref{it:f} of Lemma \ref{lem:K_j} implies \descref{(P1)}, and property \eqref{it:g} gives \descref{(P2)}. The group $M$ is a
quotient of $K_0$, which has (FA), hence $M$ has (FA) as property (FA) passes to quotients, thus \descref{(P3)} also holds. So it remains to check \descref{(P4)}.

Take any $j \in \N$ and consider any $c \in G_j\setminus G_{j-1}$ in $M$. Then, by condition \eqref{it:h},
the element $c^{-1} c^{t_jct_j^{-1}}$ will have infinite order in $K_j$ and the cyclic subgroup generated by this element will be contained in $Q_j$. Since the epimorphism $\tau_j$ is injective
on $Q_j$, we can conclude that $\tau_j(c^{-1} c^{t_jct_j^{-1}})=c^{-1} c^{a_jca_j^{-1}}$ has infinite order in $M$. Clearly
$c^{-1} c^{a_jca_j^{-1}}\in \langle G_j,G_j^{a_j c a_j^{-1}}\rangle$, thus the subgroup $\langle G_j,G_j^{a_j c a_j^{-1}}\rangle\leqslant M$ is infinite.
One can also note that the element $c (c^{-1})^{a_j^{-1}ca_j}\in \langle G_j,G_j^{a_j^{-1} c a_j}\rangle$ is a cyclic conjugate of $c^{-1} c^{a_jca_j^{-1}}$ in $M$.
Consequently, the subgroup $ \langle G_j,G_j^{a_j^{-1} c a_j}\rangle\leqslant M$ is also infinite.
Recalling that for each $n \in \N\cup \{0\}$, the subgroup $G_n \leqslant M$ is finite,
we are able to conclude that \descref{(P4)} holds.
\end{proof}

\section{Proof of the main result}\label{sec:main_proof}

We are finally prepared to prove the main result.
\begin{proof}[Proof of Theorem \ref{thm:main}] Let $M$ be the finitely generated group given by Theorem \ref{thm:M}. Then we can construct the limit group $L$ and the $\R$-tree $T$ as in Section \ref{sec:strong_conv}.
The group $L$ is finitely generated and acts on $T$ non-trivially by isometries with finite arc stabilizers by Theorem~\ref{thm:small_arc_stab},
since each $G_n \leqslant M$ is a finite group by construction (see Section~\ref{sec:M}). Moreover, $L$ has property (FA) by Lemma \ref{lem:(FA)}.

Finally, if $P$ is a finitely presented group then any epimorphism from $P$ to $L$ factors through some epimorphism $P \to L_i$ for some $i \in \N$, because $L$ is the direct limit of the groups $L_i$
(see \cite[Lemma 3.1]{Cor-Kar} for a proof of this fact). Therefore $P$ inherits from $L_i$ a non-trivial action on the Bass-Serre tree $T_i$ (corresponding to the splitting of $L_i$ as an amalgamated free product),
and thus $P$ does not have (FA).
\end{proof}

\begin{rem}\label{rem:dyadic} Recall that, by construction, each $L_i$ acts on the simplicial $\R$-tree $T_i$, $i \in \N$, where the length of an edge is set to be $1/2^{i-1}$.
Since $T_i$ converge to $T$ strongly, it is clear that their $0$-skeletons converge to a $\mathbb{D}$-tree $S$, where $\mathbb{D}\leqslant \mathbb{Q}$ is the group of dyadic rational numbers,
and $T$ is the $\R$-completion of $S$ (see \cite[Sec 1]{Gillet-Shalen} for a discussion of $\Lambda$-completions). Evidently the natural action of $L$ on $S$ is still non-trivial, thus
the pair $(L,S)$ gives an example of a finitely generated group $L$ which has property (FA), but admits a non-trivial action, without inversions, on a $\mathbb D$-tree $S$.  Since
the $\mathbb{Q}$-rank of $\mathbb{D}$ is $1$, this example shows that finite presentability is a necessary assumption in the results of Gillet and Shalen \cite[Prop. 27 or Thm. C]{Gillet-Shalen},
mentioned in the Introduction.
\end{rem}


  \end{document}